\documentclass[11pt,reqno]{amsart}
\usepackage{amsthm}
\usepackage{amssymb}
\usepackage{stackrel}
\usepackage[scr]{rsfso}
\usepackage{algorithmic} % fuer Algorithmen
\usepackage{algorithm} % Umgebung um Algorithmen herum
\usepackage{graphics}
\usepackage{graphicx}
\usepackage{tikz}
\usepackage{caption}
\usepackage{mathabx}
\usepackage{cleveref}
\usepackage[margin=1in]{geometry}
\newtheorem{proposition}{Proposition}[section]
\newtheorem{theorem}[proposition]{Theorem}
\newtheorem{corollary}[proposition]{Corollary}
\newtheorem{lemma}[proposition]{Lemma}
\newtheorem{definition}[proposition]{Definition}

\newtheorem{example}[proposition]{Example}
\newtheorem{remark}[proposition]{Remark}
\def\argmin{ \mathop{{\rm argmin}}}

\newcommand{\cl}{\mathrm{cl}\,}

\newcommand{\dom}{\mathrm{dom}\,}
\newcommand{\epi}{\mathrm{epi}\,}

\newcommand{\gph}{\mathrm{gph}\,}

\newcommand{\para}{\mathrm{par}\,}

\newcommand{\ri}{\mathrm{ri}\,}

\newcommand{\rge}{\mathrm{rge}\,}

\newcommand{\lin}{\mathrm{lin}\,}
\newcommand{\spn}{\mathrm{span}\,}

\newcommand{\p}{\partial}

\newcommand{\R}{\mathbb{R}}

\newcommand{\bX}{\mathbb{X}}
\newcommand{\bY}{\mathbb{Y}}
\newcommand{\bP}{\mathbb{P}}
\newcommand{\bU}{\mathbb{U}}
\newcommand{\bV}{\mathbb{V}}
\newcommand{\bW}{\mathbb{W}}

\newcommand{\rp}{\mathbb R\cup\{+\infty\}}

\newcommand{\cone}{\mathrm{cone}}

\newcommand{\bN}{\mathbb{N}}

\newcommand{\ip}[2]{\left\langle #1,\, #2\right\rangle}

\newcommand{\set}[2]{\left\{#1\,\left\vert\; #2\right.\right\}}

\newcommand{\bp}{{\bar{p}}}
\newcommand{\bu}{{\bar{u}}}
\newcommand{\bv}{{\bar{v}}}
\newcommand{\bx}{{\bar{x}}}
\newcommand{\by}{{\bar{y}}}
\newcommand{\bz}{{\bar{z}}}

\newcommand{\gam}{\gamma}

\newcommand{\sig}{\sigma}

\newcommand{\lam}{\lambda}
\newcommand{\Lam}{\Lambda}

\newcommand{\cL}{\mathcal{L}}

\newcommand{\st}{\quad \mbox{s.t.}\quad }

\newcommand{\vphi}{\varphi}

\date{\today}

\title[Lipschitz stability of regularized least-squares]{Lipschitz stability of  least-squares problems regularized by functions with $\mathcal{C}^2$-cone reducible conjugates}
\author{Ying Cui}
\address{Department of Industrial Engineering and Operations Research, University of California, Berkeley, Berkeley, CA
94720, USA}
\email{yingcui@berkeley.edu}

\author{Tim Hoheisel}
\address{Department of Mathematics and Statistics, McGill University, Montr\'eal, QC H3A0B9, Canada}
\email{tim.hoheisel@mcgill.ca}

\author{Tran T. A. Nghia}
\address{Department of Mathematics and Statistics, Oakland University, Rochester, MI 48309, USA}
\email{nttran@oakland.edu}

\author{Defeng Sun}
\address{Department of Applied Mathematics, The Hong Kong Polytechnic University, Hung Hom, Kowloon, Hong Kong}
\email{defeng.sun@polyu.edu.hk}

\begin{document}

\begin{abstract}
    In this paper, we  study Lipschitz continuity  of  the solution mappings of  regularized
    least-squares  problems for which the convex regularizers  have  (Fenchel) conjugates that are $\mathcal{C}^2$-cone reducible. Our approach, by using Robinson's strong  regularity on the dual problem, allows us to obtain new characterizations of Lipschitz stability that rely solely on first-order information, thus bypassing the need to explore second-order information (curvature) of the regularizer. We show that these solution mappings are automatically Lipschitz continuous around the points in question whenever they are locally single-valued.  We leverage our findings to obtain new characterizations of full stability and tilt stability for a broader class of convex additive-composite problems.  
\end{abstract}

\maketitle

\vspace{0.1in}


\section{Introduction} 
 Stability (or sensitivity) analysis in optimization is the study of how the optimal solutions (or  the optimal value) of an optimization problem varies as the problem-defining data (parameters) changes. This field has a rich and longstanding tradition anchored by the pioneering  works of Kojima \cite{Koj80, KoH84} and Robinson \cite{R80, R84, R87},  primarily focused on smooth, nonlinear programming in finite dimensions. The monograph by Bonnans and Shapiro \cite{BS00} presents the state-of-the-art treatment for problems with smooth data, in particular, in infinite dimensions. A rich  source of general theory  for nonsmooth problems, based on graphical differentiation of set-valued maps, is provided in the monograph by Dontchev and Rockafellar \cite{DR14}. 

Recently, there has been  renewed interest in stability analysis for the following nonsmooth,   convex optimization problems which occur ubiquitously in machine learning and statistical estimations:
\begin{equation}\label{p:MP}
    \min_{x\in \bX}\quad \frac{1}{2\mu}\|Ax-b\|^2+g(x),
\end{equation}
where $A: \bX\to\bY$ is a linear operator between two Euclidean spaces $\bX$ and $\bY$, $b$ is a vector in $\bY$, $\mu>0$ is the tuning parameter, and $g:\bX\to \R\cup\{+\infty\}$ is a closed, proper, convex function (usually called a regularizer in machine learning).  When the triple $(A,b,\mu)$ are considered parameters, it is important to study stability of the solution mapping of problem~\eqref{p:MP} defined by
\begin{equation}\label{eq:SM}
    S(A,b,\mu):=\argmin_{x\in\bX}\left\{\frac{1}{2\mu} \|Ax-b\|^2+g(x)\right\}.
\end{equation}
Specifically, in this paper, we aim to determine the conditions under which the solution mapping \eqref{eq:SM} is single-valued and Lipschitz continuous when there are small perturbations in the initial data  $(A,b,\mu)$ of problem ~\eqref{p:MP}. 

To the best of our knowledge, there are currently two main approaches for studying the stability of this solution mapping.    The first one  employs the toolkit of  modern variational analysis \`a la Dontchev and Rockafellar \cite{DR14}; see also the standard texts by   Mordukhovich \cite{M06}, or Rockafellar and Wets \cite{RoW98}. Some  recent contributions along these lines were made  by Berk et al. \cite{BBH23, BBH24}, Gfrerer and Outrata \cite{GfO22}, Hu et al. \cite{HYZ 24} (which strengthens the qualitative but not the quantitative results of Berk et al.), Meng et al. \cite{MWY24},  Nghia  \cite{N24}, and Vaiter et al. \cite{VDFPD17}. When the regularizer $g$ is the $\ell_1$ norm, the problem \eqref{p:MP} is known as the Lasso (Least Absolute Shrinkage and
Selection Operator) problem.  Berk et al. \cite{BBH23} provide a sufficient condition \cite[Assumption~4.3]{BBH23} at which the solution mapping $S$ is single-valued, Lipschitz continuous and directionally differentiable around a fixed triple $(\bar A, \bar b,\bar \mu).$ This condition appeared in \cite{MY12,T13} as a sufficient condition  for the uniqueness of a given solution to the Lasso problem. The approach in \cite{BBH23} relies on the well-known Mordukhovich criterion \cite[Theorem~4.18]{M06} for {\em metric regularity} and available second-order information of the $\ell_1$ norm. On the other hand, in the recent paper \cite{N24}, Nghia proposed a different approach via Robinson's {\em generalized implicit function theorem} \cite{R80} and {\em tilt stability} \cite{PR98} for more general functions $g$,  including popular regularizers such as the $\ell_1$ norm, the $\ell_1/\ell_2$ norm, and the nuclear norm. It is revealed in \cite[Theorem~3.7]{N24} that the aforementioned  sufficient condition in \cite{BBH23} is also necessary for the Lipschitz stability of the solution mapping \eqref{eq:SM} for the  Lasso problem. Some recent papers \cite{HYZ 24,MWY24} obtain impressive results about Lipschitz continuity of the set-valued solution mapping \eqref{eq:SM}  with fixed operator $A$ {\em relative} to its domain when $g$ is the $\ell_1$ norm. Their Lipschitz stability is different from ours in this paper and their approaches strongly rely on the polyhedral structure of the $\ell_1$ norm and its subdifferential mapping. When regularizer $g$ is {\em partly smooth},  Vaiter et al \cite[Theorem~1]{VDFPD17} proves a strong result that $S(\bar A,\cdot,\bar \mu)$ is continuously differentiable around $\bar b$ when $\bar x=S(\bar A,\bar b,\bar \mu)$ is a {\em strong  minimizer} of the corresponding problem, but they have to assume that $\bar b$ does not belong to a {\em transition space}. The latter condition seems to be involved and  recently \cite[Example~3.14]{N24} provides a simple example at which the solution mapping $S$ is Lipschitz continuous, but it is not  continuously differentiable.

The second approach,  spearheaded by Bolte  et al. \cite{BLPS21}, is to  rewrite the optimality conditions as a (locally) Lipschitz equation, say, by using proximal operators, and to employ the machinery of implicit differentiation with {\em conservative Jacobians} \cite{BP21}. The conservative Jacobians here serve as a versatile tool which, among other things, bypasses the failure of the chain rule for Clarke Jacobians. The ideas mentioned here can be extended to (maximally) monotone inclusions (which are not  necessarily optimality conditions of a convex optimization problem) \cite{BPS24}.

In this paper, we mainly study Lipschitz stability of solution mapping \eqref{eq:SM} with a fresh approach via {\em strong regularity} \cite{R80} via the dual problem. Strong regularity was introduced by Robinson in the landmark paper \cite{R80} for {\em generalized equations} with applications to {\em nonlinear programming} to study  Lipschitz stability of the {\em Lagrange system}. Robinson's strong regularity can be characterized by the well-known {\em Linear Independence Constraint Qualification} and the {\em Strong Second-Order Sufficient Condition} for nonlinear programming; see, e.g., \cite{R80,BS00,KK02,DR14} for more historic discussions and its immense applications to optimization theory and algorithms. In \cite[Theorem~5.6]{MNR15},  Mordukhovich, Nghia and Rockafellar give  full characterizations of strong regularity for more general parametric constrained optimization problems in the following setting:
\begin{equation}\label{p:OPT0}
    \min_{u\in \bU}\quad \varphi(u,p)\quad\mbox{subject to}\quad G(u,p)\in \Theta,  
\end{equation}
with Euclidean spaces $\bU,\bV, \bP$ ($\bP$ being the parameter space),  a closed and convex set $\Theta\subset\bV$, and twice continuously differentiable  cost function $\varphi:\bU\times \bP\to \R$  and constrained mapping $G:\bU\times\bP\to \bV$. It is worth noting here that \cite{MNR15} crucially assumes that the constrained set $\Theta$ is $\mathcal{C}^2$-{\em reducible}  in the sense of \cite[Definition~3.135]{BS00} that allows them to cover some previous works on the strong regularity for non-polyhedral settings such as  the {\em second-order cone programming} \cite{BR05} and the {\em semidefinite programming} \cite{S06}. Their characterizations are based  on the so-called {\em Partial Constraint Nondegeneracy} and the {\em Generalized Strong Second-Order Sufficient Condition} that will be recalled in our Subsection~2.3.

A simple, but important observation of our paper is that the solution mapping \eqref{eq:SM} is, indeed,  the mapping of  (partial)  Lagrange multipliers of the following constrained problem 
\begin{equation}\label{p:Dut}
\quad \min_{(y,t)\in\bY\times \R}\quad  \frac{\mu}{2}\|y\|^2-\ip{b}{y}+t\quad \mbox{subject to}\quad (A^*y,t)\in \epi g^*,
\end{equation} 
where $\epi g^*$  stands for the epigraph of the Fenchel conjugate function of $g$. 
This problem is in the format of \eqref{p:OPT0} and equivalent to the Fenchel-Rockafellar dual of problem \eqref{p:MP}
\begin{equation*}\label{p:Du0}
\quad \min_{y\in\bY}\quad  \frac{\mu}{2}\|y\|^2-\ip{b}{y}+g^*(A^*y).
\end{equation*} 
This unveils  the connection between the strong regularity of problem \eqref{p:Dut} and the Lipschitz stability of solution mapping \eqref{eq:SM} of problem~\eqref{p:MP}. As mentioned above, in order to apply \cite[Theorem~5.6]{MNR15}, we need to suppose that the constrained set of problem \eqref{p:Dut}, i.e,  the epigraph of the conjugate function $ g^*$, is $\mathcal{C}^2$-{\em cone reducible}. This assumption is not restrictive: it automatically holds for many important regularizers in optimization such as the $
\ell_1$ norm, the $\ell_1/\ell_2$ norm, the nuclear norm, and many more  spectral functions \cite{CDZ17} and support functions; see also our Remark~\ref{rem:Ex}.

While applying the theory of strong regularity in \cite{MNR15} to the dual problem \eqref{p:Dut}, we observe that the Generalized Second-Order Sufficient Condition holds automatically and the Partial Constraint Nondegeneracy can be simplified with first-order information on the regularizer $g$.  This  significantly distinguishes our paper from  \cite{BBH23,N24,VDFPD17} that directly work on the primal problem and  needs more advanced {\sl second-order} computation on the regularizer. In particular, we show that the solution mapping $S$ is single-valued and Lipschitz continuous around a fixed triple $(\bar A,\bar b,\bar \mu)\in \mathcal{L}(\bX,\bY)\times \bY\times \R_{++}$  of initial data  for $\bx\in S(\bar A,\bar b,\bar \mu)$ if and only if 
\[
\ker \bar A\cap {\rm par}\,  \partial g^*(\bar z)=\{0\}\quad \mbox{with}\quad \bar z:=-\frac{1}{\bar \mu}\bar A^*(\bar A\bar x-\bar b), 
\]
where ${\rm par}\,  \partial g^*(\bar z)$ denotes the {\em parallel subspace} of the subdifferential set $\partial g^*(\bar z)$. This characterization recovers some similar results obtained in \cite{BBH23,N24} for special cases of regularizers such as the $\ell_1$ norm, the $\ell_1/\ell_2$ norm, and the nuclear norm. Although our technique relies on second-order variation analysis, our condition above does not need any second-order information of $g$.

Another key aspect of our paper studies characterizations of tilt stability \cite{PR98} and full stability \cite{LPR00}, both of which also relate to Lipschitz stability of the solution mapping of  the following optimization problem
\begin{equation}\label{p:PP}
    \min_{x\in \bX}\quad f(x,p)+g(x),
\end{equation}
where $f:\bX\times \bP\to \R$ is a twice continuously differentiable function that is convex with respect to $x\in \bX$. This problem is more general than \eqref{p:MP} at which the basic parameter $p$ is considered as the triple parameter $(A,b,\mu)$.  In our framework,  full stability, as introduced by Poliquin, Levy, and Rockafellar \cite{LPR00},  mainly analyzes  Lipschitz continuity of the (local) optimal solution mapping of
\begin{equation}\label{p:FS}
    \min_{x\in \bX}\quad f(x,p)+g(x)-\ip{v}{x},
\end{equation}  
with respect to both  basic parameter $p\in \bP$ and  tilt (linear) parameter $v\in \bX$; while tilt stability is a particular case when the basic perturbation $p$ is fixed. Both tilt stability and full stability are characterized by using different advanced nonsmooth second-order structures that may be  difficult to compute \cite{GfO22,LPR00,MNR15,N24,PR98}. In Section~4, by using our results on Lipschitz stability, we achieve a new pointwise characterization for both full stability and tilt stability of problem~\eqref{p:PP}. Unlike  other characterizations of full stability and tilt stability, our pointwise condition uses first-order information of the function $g$. 
Specifically, we show that $\bar x\in \bX$ is a full (or tilt) stable optimal solution of problem \eqref{p:PP} for parameter $\bp\in \bP$ if and only if $-\nabla_x f(\bx,\bp)\in \partial g(\bar x)$ and 
\[
\ker \nabla^2_{xx} f(\bx,\bp)\cap {\rm par}\,  \partial g^*(-\nabla_x f(\bx,\bp))=\{0\};
\]
see our Theorem~\ref{thm:Lips} for further details. This condition still needs second-order information of the function $f$, but  is easy to compute, as $f$ is twice differentiable. The crucial part of this condition is that we do not need to calculate any second-order structures for the nonsmooth function $g$. Consequently, it is also a sufficient condition for Lipschitz stability of the solution mapping of problem \eqref{p:PP}. However, unlike the regularized least-squares problem~\eqref{p:MP}, this condition is not necessary for the aforementioned Lipschitz stability of problem~\eqref{p:PP}. 

\medskip

\noindent
{\em Notation:} Throughout the paper,  $\mathcal{L}(\bX,\bY)$ is the space of all linear operators from the Euclidean space $\bX$ to the Euclidean space $\bY$. For  $A\in \mathcal{L}(\bX,\bY)$, we denote its {\em adjoint} by $A^*$, its {\em kernel} (or {\em null space}) by $\ker A$ and its {\em range} by $\rge A$.
For a Fr{\'e}chet differentiable map $G:\bX \to \bY$, we write ${\mathcal J} G(x)\in \bY \times \bX$ as the Jacobian matrix of $G$ at $x\in \bX$. When $\bY = \mathbb{R}$, we have ${\mathcal J} G(x) = \nabla G(x)^*$ for any $x\in \bX$.  

\section{Preliminaries}

Let $\bX$ be a Euclidean space, i.e., a finite-dimensional real vector space with inner product denoted by $\ip{\cdot}{\cdot}$. 
%We identify the  dual space $\bX^*$ with $\bX$, and henceforth denote it as such.  
The (Euclidean) norm on $\bX$ induced by the ambient inner product  is denoted by $\|\cdot\|$. The closed ball in $\bX$ with center $x\in \bX$ and radius $r>0$ is denoted by $\mathbb{B}_r(x)$.

\subsection{Tools from  convex analysis}\label{sec:convex}
For a function $f:\bX\to\rp$,  its {\em epigraph} is the set given by $\epi f:=\set{(x,\alpha)\in\bX\times \R}{f(x)\leq \alpha}$ while its {\em domain} is  the set $\dom f:=\set{x\in\bX}{f(x)<+\infty}$. We call $f$ {\em proper} if $\dom f$ is nonempty. We say that $f$ is {\em convex} if $\epi f$ is convex, and we call it {\em closed} or {\em lower semicontinuous (lsc)} if $\epi f$ is closed. The {\em (Fenchel) conjugate} of a proper $f$ is the function $f^*:X\to\rp$, given by 
$
f^*(y)=\sup_{y\in \bX}\{\ip{y}{x}-f(x)\}.
$
Its {\em subdifferential} at $\bar x\in \bX$ is the set $\partial f(\bar x)=\set{y\in \bX}{f(\bar x)+\ip{y}{x-\bar x}\leq f(x)\; \forall x\in \bX}$. When $f$ is closed, proper, convex, throughout the paper we use the following relation frequently:
\[
v\in \p f(x)\iff x\in \p f^*(v).
\]
%Let us recall next a few geometric structures used in the paper.
For a convex set  $\Omega\subset \bX$,
its  {\em relative interior} is given by
\begin{equation}\label{def:Ri}
\ri \Omega:=\set{x\in \Omega}{\exists\, \epsilon>0:\; \mathbb{B}_\epsilon(x)\cap{\rm aff}\,\Omega\subset\Omega},
\end{equation}
where ${\rm aff}\, \Omega$ is the {\em affine hull} of $\Omega$. The {\em conical hull}  of $\Omega$ is $\cone\; \Omega:=\R_+\Omega$, the smallest (convex) cone containing $\Omega$. The {\em span} of $\Omega$  is denoted $\spn \Omega$. The {\em parallel subspace} of   $\Omega$ is  defined, for any $x\in \Omega$, by
\begin{equation}\label{def:Par}
\para\Omega  :=    {\rm span}\, (\Omega-x).
\end{equation}
 Note that   $\para\Omega ={\rm aff}\, \Omega-{\rm aff}\,\Omega={\rm aff}\, \Omega-x$ for any $x\in \Omega$. Thus for any $x\in \ri \Omega$, we have
\begin{equation}\label{eq:ParC}
\para\Omega  =    \cone\, (\Omega-x).
\end{equation}
The {\em lineality space} of $\Omega$, denoted by ${\rm lin}\,\Omega$, is the largest subspace in $\Omega.$  Moreover, the {\em horizon cone}  of $\Omega$ is 
\begin{equation}\label{def:Hor}
    \Omega^\infty:=\set{v\in \bX}{\exists\, \{x_k\in \Omega\}, \{t_k\}\downarrow 0:\; t_kx_k\to v},
\end{equation}
which is a closed cone, and convex as $\Omega$ is.  In particular, for a  convex function $f$ and   $\bar y\in \p f(\bar x)$, we have
\begin{equation}\label{eq:Horiz2}
    \p f(\bar x)^\infty\subset\cone(\p f(\bar x)-\bar y),
\end{equation}
cf. \cite[Theorem 3.6]{RoW98}.   The {\em normal cone} of $\Omega$ at $\bx$ is 
\begin{equation}\label{def:Nor1}
N_\Omega(\bx)=\set{y\in\bX}{\ip{y}{x-\bx}\le 0\;\forall x\in \Omega},
\end{equation}
which is  the subdifferential of the (convex) {\em indicator function} $\delta_\Omega$ of $\Omega$ at $\bx$, which is defined by $\delta_\Omega(x)=0$ if  $x\in \Omega$ and $\delta_\Omega(x)=+\infty$ otherwise. 
The following result, which is useful to our study, links many of the above mentioned objects. 

 \begin{proposition}\label{prop:Aux} Let $f:\bX\to\rp$ be a closed, proper, convex function and $\bar x\in\bX$ such that $\p f(\bar x)\neq \emptyset$. Then 
\[
\bX\times \{0\}\cap \spn N_{\epi f}(\bar x,f(\bar x))= \para \p f(\bar x)\times \{0\}.
\]
\end{proposition}
\begin{proof}  First,  observe by \cite[Theorem 8.9]{RoW98} that
\begin{equation}\label{eq:Sepi}
N_{\epi f}(\bar x,f(\bar x))=\set{\sigma(v,-1)}{v\in \p f(\bar x),\;\sigma>0}\bigcup \p f(\bar x)^\infty\times \{0\}.
\end{equation}
Now, let $(v,0)\in \bX\times \{0\}\cap \spn N_{\epi f}(\bar x,f(\bar x))$. In view of \eqref{eq:Sepi}, we find scalars $\lambda_1,\dots, \lambda_r\in \R$,  $\mu_1,\dots,\mu_s\in \R$, and vectors $v_1,\dots, v_r\in \p f(\bar x)$, $w_1,\dots, w_s\in \p f(\bar v)^\infty$ such that 
\[
(v,0)=\sum_{i=1}^r\lambda_i(v_i,-1) +\sum_{j=1}^s \mu_j (w_j,0). 
\]
It follows that $\sum_{i=1}^r\lambda_i=0$, i.e., $\lambda_r=-\sum_{i=1}^{r-1}\lambda_i$.
With the previous identity, we thus obtain
\[
v=\sum_{i=1}^{r-1}\lambda_i(v_i-v_r)+\sum_{j=1}^s \mu_j w_j\subset \para \p f(\bar x)+\spn \p f(\bar x)^\infty.
\]
For any $\bar y\in \p f(\bar x)$, note from \eqref{eq:Horiz2} that
\[
\spn \p f(\bar x)^\infty\subset \spn (\p f(\bar x)-\bar y)=\para \p f(\bar x).
\]
Therefore, $v\in \para \p f(\bar x)$, which shows the desired inclusion. 

Conversely, take $(v,0)\in \para \p f(\bar x)\times \{0\}$. Then there exist $v_1,\dots, v_r\in \p f(\bar x)$ and $\lambda_1,\dots, \lambda_{r-1}\in \R$ such that 
\[
v=\sum_{i=1}^{r-1}\lambda_i (v_i-v_r). 
\]
Setting $\lambda_r:=-\sum_{i=1}^{r-1}\lambda_i$, in view of \eqref{eq:Sepi},   we find that 
\[
(v,0)=\sum_{i=1}^r\lambda_i(v_i,-1)\in \spn N_{\epi f}(\bar x,f(\bar x)).
\]
This concludes the proof. 
\end{proof}

\subsection{Tools from set-valued analysis} \label{sect:SVA}
Let $S\subset \bX$ and $\bx\in S$. We write $x_k\stackrel[S]{}{\to} \bar x$ if $\{x_k\}\to \bar x$ and $x_k\in S$ for all $k\in\bN$. The {\em regular normal cone} to $S$ at $\bar x$ is the (convex) cone defined by 
\[
\hat N_S(\bar x):=\set{y\in \bX}{\limsup_{x\stackrel[S]{}{\to} \bar x}\frac{\ip{y}{x-\bar x}}{\|x-\bar x\|}\leq 0}.
\]
It is often expedient to realize that $y\in \hat N_S(\bar x)$ if and only if 
\begin{equation}\label{eq:RegularEps}
\forall\, \varepsilon>0\;\exists\, \delta>0\;\forall\, x\in B_\delta(\bar x)\cap S:\; \ip{y}{x-\bar x}\leq \varepsilon \|x-\bar x\|.
\end{equation}
The  {\em limiting normal cone} \cite{M06,RoW98} to $S$ at $\bar x\in S$ is defined by
\begin{equation}\label{def:Nor2}
 N_S(\bx):=\set{y\in \bX}{\exists\, \{x_k\}\stackrel[S]{}{\to} \bar x,\; \{y_k\in \hat N_S(x_k)\}: \; y_k\to y}, 
\end{equation}
i.e., it is the {\em outer limit} of the regular normal cone. 
When $S$ is convex,  both the regular and the limiting normal  cone reduce to the object in \eqref{def:Nor1}.
The {\em tangent cone} to $S$ at $\bx$ is defined by
\begin{equation}\label{eq:Tan}
T_S(\bx):=\set{w\in \bX}{\exists\, \{t_k\}\downarrow 0, \{w_k\}\to w: \bx+t_kw_k\in S}.
\end{equation}

 Let us recall next the limiting and singular subdifferentials of an extended real-valued (possibly nonconvex) function $f:\bX\to \R\cup\{\infty\}$ that is lower semicontinuous at some $\bx\in \dom f$; see, e.g.,  \cite[Definition~1.77]{M06}. The {\em limiting subdifferential} of $f$ at $\bx$ is defined by
\begin{equation}\label{eq:limS}
\partial f(\bx):=\{y\in \bX|\; (y,-1)\in N_{{\rm epi}\, f}(\bx, f(\bx))\},
\end{equation}
which returns the subdifferential of convex analysis defined in Section~\ref{sec:convex} when $f$ is a convex function. 
Moreover, the {\em singular subdifferential} of $f$ at $\bx$ is defined by
\begin{equation}\label{eq:sinS}
\partial^\infty f(\bx):=\{y\in \bX|\; (y,0)\in N_{{\rm epi}\, f}(\bx, f(\bx))\},
\end{equation}
which is the horizon cone of $\partial f(\bx)$ when $f$ is convex, cf. \cite[Proposition~8.12]{RoW98}.

An  important tool from set-valued differentiation to our study is the  {\em (limiting) coderivative} \cite{M06,RoW98}. Let $F:\bX\rightrightarrows \bY$ be a set-valued mapping between two Euclidean spaces.  The {\em graph} of $F$ is defined by
\[
\gph F:=\set{(x,y)\in \bX\times\bY}{y\in F(x)}.
\]
The (limiting) coderivative of $F$ at  $\bx$ for  $\by\in  F(\bx)$ is the set-valued mapping $D^*F(\bx|\,\by):Y\rightrightarrows X$ defined by
\begin{equation}
    D^* F(\bx|\,\by)(w):=\set{z\in \bX}{(z,-w)\in N_{\gph F}(\bx,\by)}.
\end{equation}
The set-valued mapping $F$ is said to have a {\em single-valued (and Lipschitz continuous) localization at $\bx$ for $\by$}  (cf. \cite{BS00,DR14,S06,MNR15, R80}) if there exist open neighborhoods $U$ of $\bx$ and $V$ of $\by$ such that the {\em localization} $s:U\to V$ of $F$ with $\gph s=\gph F\cap (U\times V)$ is single-valued (and Lipschitz continuous). 

 We are interested in single-valued locally Lipschitz continuous functions. To make the transition from  single-valued localizations to actual single-valuedness for convex-valued mapping, the following result is useful.
\begin{lemma}\label{lem:SingleVal} Let $F:\bX\rightrightarrows \bY$ be a set-valued mapping, and let $(\bar x,\bar y)\in \gph F$. Suppose that the set $F(x)$ is convex when $x$ is around $\bx$.  Then the following are equivalent:
\begin{itemize}
\item[(i)] $F$ has a  single-valued  localization at $\bx$ for $\by$;
\item[(ii)] $F$ is single-valued around $\bx$.
\end{itemize}
\end{lemma}
\begin{proof} {[}(i)$\Rightarrow$ (ii){]}: If there exist  neighborhoods $U$ of $\bx$ and $V$ of $\by$ such that the {\em localization} $s:U\to V$ of $F$ with $\gph s=\gph F\cap (U\times V)$ is single-valued, we may suppose w.l.o.g. that $F(x)$ is convex for any $x\in U$ (otherwise shrink the neighborhood). Pick any $y\in F(x)$.  As $F(x)$ is convex and $s(x)\in F(x)$, we find some $\lam\in (0,1)$ sufficiently small such that $s(x)+\lam(y-s(x))\in F(x)\cap V$, which implies that $s(x)+\lam(y-s(x))=s(x)$, i.e., $y=s(x)$. Hence $F(x)=s(x)$ for $x\in U$, which means $F$ is single-valued around $\bx$. The opposite implication is straightforward. 
\end{proof}

\begin{corollary}\label{cor:SingleVal} Let $h:\bX\times \bP\to \R\cup\{\infty\}$ be a proper function such that $h(\cdot,p)$ is convex for all $p\in \bP$. Define $S:\bP\to \bX$ by
\[
S(p):=\argmin_{x\in \bX} h(x,p).
\]
Let $(\bar p,\bx)\in \gph S$. Then  $S$ has a single-valued (Lipschitz) localization at $\bar p$ for $\bar x$ if and only if $S$ is single-valued (and locally Lipschitz) around  $\bar x$.

\end{corollary}\begin{proof}
    This follows immediately from \Cref{lem:SingleVal}, since $S$ is convex-valued.
\end{proof}

% {\color{purple}Moreover, if additionally $F(x)$ is convex for $x$ around $\bx$,  $F$ has a  single-valued and Lipschitz continuous localization at $\bx$ for $\by$ if and only if it is single-valued and Lipschitz continuous around $\bx$ with $F(\bx)=\by$.  The ``$\Rightarrow$'' part is trivial. For the ``$\Leftarrow$'' part,  if there exist  neighborhoods $U$ of $\bx$ and $V$ of $\by$ such that the {\em localization} $s:U\to V$ of $F$ with $\gph s=\gph F\cap (U\times V)$ is single-valued and Lipschitz continuous, we may suppose that $F(x)$ is convex for any $x\in U$ by shrinking the neighborhood $U$  if necessary. Pick any $y\in F(x)$, as $F(x)$ is convex and $s(x)\in F(x)$, we find some $\lam\in (0,1)$ sufficiently small such that $s(x)+\lam(y-s(x))\in F(x)\cap V$, which implies that $s(x)+\lam(y-s(x))=s(x)$, i.e., $y=s(x)$. Hence $F(x)=s(x)$ for $x\in U$, which means $F$ is single-valued and Lipschitz continuous around $\bx$ with $F(\bx)=\by$.} 

\subsection{Tools from parametric optimization}

For  Euclidean spaces $\bU$, $\bV$, and $\bP$,  consider the following parametric optimization problem 
\begin{equation}\label{p:OPT}
{\rm OPT}(p):\quad \min_{u\in \bU} \quad \varphi(u,p)\st  G(u,p)\in \Theta,
\end{equation}
where $\Theta\subset\bV$ is a closed and convex set,  and $\varphi:\bU\times \bP\to \R$  and $G:\bU\times \bP\to \bV$ are twice continuously differentiable functions. {\em Robinson's constraint qualification} is said to hold for {\rm OPT}($\bar p$) at some $\bu\in \bU$ with $G(\bu,\bp)\in \Theta$ (i.e., $\bar u$ is feasible for OPT($\bar p$)) if  
\begin{equation}\label{eq:RCQ}
0\in {\rm int}\, (G(\bu,\bp)+\rge {\mathcal J}_u G(\bu,\bp)-\Theta).
\end{equation}
The {\em Lagrangian} of problem OPT($p$) is $\mathscr{L}:\bU\times \bP \times \bV\to \R$ given by
\begin{equation}\label{eq:Lag}
    \mathscr{L}(u,p,\lam):=\varphi(u,p)+\ip{\lam}{G(u,p)}\quad \mbox{for}\quad (u,p,\lam)\in \bU\times \bP \times \bV.
\end{equation}
A feasible point $u\in\bU$ is called a {\em stationary} point of OPT($p$) if there exists a  {\em Lagrange multiplier} $\lam\in \bV$ satisfying  
\begin{equation}\label{eq:LM}
    0=\nabla_u\mathscr{L}(u,p,\lam)\quad \mbox{and}\quad \lam\in N_\Theta(G(u,p)). 
\end{equation}
We define 
$
\Lam(u,p):=\set{\lam \in \bV}{0=\nabla_u\mathscr{L}(u,p,\lam),\; \lam\in N_\Theta(G(u,p))}
$, the set of all Lagrange multipliers at $(u,p)$. 

Observe that  the system \eqref{eq:LM} can be written as the  {\em generalized equation} \cite{R80}:
\begin{equation}\label{eq:GE}
 {\rm GE}(p): \quad 0\in \begin{pmatrix}\nabla_u \vphi(u,p)+{\mathcal J}_u G(u,p)^*\lam\\-G(u,p)\end{pmatrix}+\begin{pmatrix}0\\N_\Theta^{-1}(\lam)\end{pmatrix}.  
\end{equation}
Let  $H:\bP\rightrightarrows\bU\times\bV$ be its {\em solution mapping}, i.e.,
$
H(p)=\set{(u,\lam)}{(u,p,\lam) \;\text{satisfies\;\eqref{eq:GE}}}.
$
Given $\bar p\in \bP  $ and $(\bar u,\bar \lambda)\in H(\bar p)$, it is an immediate question under which conditions  $H$ has a single-valued and  Lipschitz continuous localization at $\bp$ for $(\bu,\bar \lam)$ (in the sense of the definition at the end of the previous section). To answer this question, Robinson \cite{R80} came up with an original idea of considering the {\em linearized generalized equation} of \eqref{eq:GE} below:
\begin{equation}\label{eq:LGE}
\delta\in \begin{pmatrix}
0\\-G(\bu,\bar p)\end{pmatrix}    
+\begin{pmatrix}\nabla^2_{uu}\mathscr{L}(\bu,\bar p,\bar\lam)(u-\bu)+{\mathcal J}_u G(\bu,\bar p)^*(\lam-\bar\lam)\\-{\mathcal J}_uG(\bu,\bar p)(u-\bu)\end{pmatrix}+\begin{pmatrix}0\\N_\Theta^{-1}(\lam)\end{pmatrix}  
\end{equation}
with $\delta\in \bU\times \bV$. Let us set the solution mapping of this  generalized equation by $K:\bU\times \bV\rightrightarrows \bU\times \bV$ with  $K(\delta)=\{(u,\lam)|\; (u,\lam,\delta) \;\text{satisfies\;\eqref{eq:LGE}}\}$ for any $\delta\in  \bU\times \bV$. If $K$ has a single-valued and Lipschitz continuous localization around $\bar \delta=0$ for $(\bu,\bar \lam)$, then $H$ also has a single-valued and Lipschitz continuous localization around $\bp$ for $(\bu, \bar \lam)$. This is usually referred to as Robinson's implicit function theorem; see the original result by Robinson \cite[Theorem~2.1]{R80}, and its extension \cite[Theorem~2B.5]{DR14} that can be applied directly to our above framework. Lipschitz continuity of the solution mapping of system~\eqref{eq:LGE} is known as  {\em Robinson's strong regularity} defined precisely next.

\begin{definition}[Robinson's strong regularity]\label{def:RSR} Suppose that $(\bu,\bar \lam)$ is a solution of the generalized equation  ${\rm GE}(\bar p)$ in  \eqref{eq:GE}, i.e., $\bu$ is a stationary point of problem ${\rm OPT}(\bar p)$ in \eqref{p:OPT} and $\bar \lam$ is a corresponding Lagrange multiplier from \eqref{eq:LM}.
 We say that the generalized equation ${\rm GE}(\bar p)$  is strongly regular at $(\bu,\bar \lam)$ if  the solution mapping $K$ of  the linearized  generalized equation \eqref{eq:LGE}
has a single-valued and Lipschitz continuous localization around $0\in \bU\times \bV$ for $(\bu,\bar\lam)\in \bU\times \bV$. 
\end{definition}
\noindent
Robinson's strong regularity was introduced in \cite{R80} for generalized equations with applications to nonlinear programming at which  $\bU=\R^n$, $\bP=\R^d$,  $\bV=\R^{m_1+m_2}$, and $\Theta=\R^{m_1}_-\times\{0_{m_2}\}\subset \R^{m_1+m_2}$. Particularly,  Robinson showed in \cite[Theorem~4.1]{R80} that the well-known {\em Linear Independence Constraint Qualification} (LICQ) of ${\mathcal J}_u G(\bu,\bar p)$ and the so-called {\em Strong Second-Order Sufficient Condition} (SSOSC) that is also proposed by Robinson are sufficient for  strong regularity of the corresponding generalized equation \eqref{eq:GE} at $(\bar u, \bar \lam)$. If additionally $\bar u$ is an optimal solution of problem ${\rm OPT}(\bar p)$ in \eqref{p:OPT}, these two conditions are also necessary for strong regularity; see, e.g., \cite[Theorem~8.10]{KK02} and also \cite{BS00,DR96,DR14} for further historic discussions about strong regularity. For semidefinite programming (SDP), Robinson's strong regularity is characterized via different conditions including the so-called {\em constraint nondegeneracy} and the {\em SDP-Strong Second-Order Sufficient Condition} by Sun \cite{S06}. For more general constrained programming as in \eqref{p:OPT},  it is shown by Mordukhovich, Nghia, and Rockafellar in \cite{MNR15} that Robinson's strong regularity is equivalent to the {\em Partial Constraint Nondegeneracy}, which originated from \cite{R84} and the {\em generalized SSOSC} recalled later, provided that the set $\Theta$ is {\em $\mathcal{C}^2$-cone reducible}; see, e.g.,  \cite[Section~3.4.4]{BS00}. Let us recall these definitions here.

\begin{definition}[Partial Constraint Nondegeneracy]\label{def:ND}
Let $\bar p\in\bP$ and let $\bar u$ be feasible for ${\rm OPT}(\bar p)$, i.e., $G(\bar u,\bar p)\in\Theta$.
We say that the point $\bar u$ is {\em partially  nondegenerate} for $G$ with respect to $\Theta$   if 
\begin{equation}\label{eq:ND0}
\rge {\mathcal J}_u G(\bu, \bar p)+\lin T_\Theta(G(\bu,\bar p))=\bV,
\end{equation}
where $\lin T_\Theta(G(\bu,\bar p))$ is the lineality space of $T_\Theta(G(\bu,\bar p))$, the largest subspace in the tangent cone $T_\Theta(G(\bu,\bar p))$.
\end{definition}

\begin{remark}[Constraint nondegeneracy]\label{rem:ND} (a)   We note that for the case of nonlinear programming when  $\bU=\R^n$, $\bP=\R^d$,  $\bV=\R^{m_1+m_2}$, and $\Theta=\R^{m_1}_-\times\{0_{m_2}\}\subset \R^{m_1+m_2}$, and $G(\bar x)\in \Theta$,  the constraint nondegeneracy at $\bar u$   holds if and only if the {\em linear independence constraint qualification (LICQ)} holds at $\bar u$. 
\smallskip

\noindent
(b) A dual formulation of \eqref{eq:ND0} reads
\begin{equation}\label{EQ:DND}
\ker {\mathcal J}_uG(\bu, \bar p)^*\bigcap \spn N_{\Theta}(G(\bu,\bar p))=\{0\}.
\end{equation}

\end{remark}
Recall that a cone $C\subset \bX$ is called {\em pointed} if $C\cap (-C)=\{0\}$.

\begin{definition}[$\mathcal{C}^2$-cone reducible sets and functions]\label{def:Redu} The set $\Theta\subset\bV$ is said to be $\mathcal{C}^2$-{\em cone reducible} at $\bv$ if there exist a neighborhood $V$ of $\bv$, a pointed, closed, convex cone  $C$ in some Euclidean space $\bW$, and a $\mathcal{C}^2$-smooth mapping $h:V\to\bW$ such that $h(\bv)=0$, ${\mathcal J} h(\bv)\in \cL(\bV,\bW)$ is surjective, and that $\Theta\cap V=\{v\in V|\;h(v)\in C\}$.

A closed, proper, convex function $f:\bU\to \R\cup\{+\infty\}$ is  called $\mathcal{C}^2$-cone reducible at $\bu\in \dom f$ if its epigraph $\epi f$ is   $\mathcal{C}^2$-cone reducible at $(\bu,f(\bu))$. We call the function $f$ to be $\mathcal{C}^2$-cone reducible, if it is $\mathcal{C}^2$-cone reducible at any point on its domain. 
\end{definition}

We point out that any set $\Theta$ is $\mathcal{C}^2$-cone reducible at any $\bar v\in {\rm int}\, \Theta$, as we may choose a neighborhood $V\subset \Theta$ of $\bv$, $\bW=\{0\}$, $C=\{0\}$, and $h(v)\equiv0$ in the above definition.  
%Thus it is more interesting to question whether $\Theta$ is $\mathcal{C}^2$-cone reducible at $\bv$ on the boundary of $\Theta$.  
Many important convex sets in optimization are $\mathcal{C}^2$-cone reducible including  (convex) polyhedra \cite[Example~3.139]{BS00}, the positive semidefinite cone \cite[Example~3.140]{BS00}, or  the second-order (or Lorentz) cone \cite[Lemma~15]{BR05}. The Cartesian product of many $\mathcal{C}^2$-cone reducible convex sets is also $\mathcal{C}^2$-cone reducible; see \cite[Proposition~3.1]{Sh03}. Moreover, it is shown in \cite[Proposition 10]{CDZ17} that many {\em spectral functions} including the nuclear norm and the spectral norm are also $\mathcal{C}^2$-cone reducible.

The following result \cite[Theorem~5.6]{MNR15} provides a full characterization of Robinson's strong regularity for $\mathcal{C}^2$-cone reducible programming.

\begin{theorem}[Characterization of Robinson's strong regularity]\label{thm:RSR} Let $\bu\in \bU$ be  a stationary point of problem ${\rm OPT}(\bar p)$ in \eqref{p:OPT} with  $\bar \lam\in \Lam(\bu,\bar p)$. Suppose that the set $\Theta$ is $\mathcal{C}^2$-cone  reducible at $G(\bu,\bp)$ and that Robinson's constraint qualification \eqref{eq:RCQ} holds  at $\bar u$. Then the  following are equivalent: 
\begin{itemize}
\item[{\bf (i)}]    The generalized equation \eqref{eq:GE} is  strongly regular at $(\bu,\bar\lam)$ and $\bu$ is a local minimizer of problem ${\rm OPT}(\bar p)$;

    \item[{\bf (ii)}]  $\bu$ is  partially nondegenerate for $G$ with respect to $\Theta$ and the following so-called {\em generalized strong second-order sufficient  condition (GSSOSC)} holds 
\begin{equation}\label{eq:GSSOSC}
\ip{w}{\nabla^2_{uu}\mathscr{L}(\bu,\bar p,\bar\lam)w}+\inf\Big\{\ip{z}{{\mathcal J}_u G(\bu,\bar p)w}\big|\;z\in D^*N_\Theta(G(\bu,\bp)|\,\bar\lam)\big({\mathcal J}_u G(\bu,\bar p)w\big)\Big\}>0
\end{equation}
for any $w\neq 0$.
\end{itemize}
\end{theorem}
\noindent
In the framework of nonlinear programming, the GSSOSC is exactly the classical SSOSC of Robinson \cite{R80}. For the case $\Theta=\mathbb{S}^n_+$, the cone of of positive semi-definite matrices, it also recovers the SDP-SSOSC \cite{S06}; see \cite{MNR15} for further details and discussions. We point out that verifying GSSOSC \eqref{eq:GSSOSC} is not trivial due to the lack of explicit calculus for the coderivative of the normal cone mapping.  In Section~\ref{Lip}, however,  we will show that GSSOSC holds automatically for the  problems in the center of our study.

\section{Lipschitz stability of least-squares problems with $C^2$-cone reducible dual regularizes}\label{Lip}
\noindent
In this section, we consider the following parametric optimization problem 
\begin{equation}\label{eq:P}
{\rm P}(A,b,\mu): \quad\min_{x\in\bX}\quad  \frac{1}{2\mu} \|Ax-b\|^2+g(x),
\end{equation}
where $\bX, \bY$ are two Euclidean spaces,  $A\in \cL(\bX,\bY)$, $b\in \bY$, $\mu>0$ are treated as  parameters  and $g:\bX\to\rp$ is a closed, proper, convex function. Define the solution mapping $S:\cL(\bX,\bY)\times \bY\times \R_{++}\to \bX$ of problem \eqref{eq:P} by 
\begin{equation}\label{eq:Sp}
S(A,b,\mu):=\argmin_{x\in\bX}\left\{\frac{1}{2\mu} \|Ax-b\|^2+g(x)\right\}.
\end{equation}
Our main goal in this section is to study sufficient and necessary conditions for this solution mapping to be (locally) Lipschitz around some fixed triple $(\bar A,\bar b,\bar \mu) \in \cL(\bX,\bY)\times \bY\times \R_{++}$. 
Given $v\in \bX$, the  {\em tilt-perturbed} problem corresponding to P$(A,b,\mu)$ reads
\begin{equation}\label{eq:Q}
Q(A,b,\mu,v): \quad\min_{x\in\bX}\quad  \frac{1}{2\mu} \|Ax-b\|^2+g(x)-\ip{v}{x}.
\end{equation} 
Its  solution mapping $\widehat S:\cL(\bX,\bY)\times \bY\times \R_{++}\times \bX\to \bX$ is defined by
\begin{equation}\label{eq:SHat}
  \widehat S(A,b,\mu,v)=\argmin_{x\in \bX} \left\{\frac{1}{2\mu} \|Ax-b\|^2+g(x)-\ip{v}{x}\right\}.  
\end{equation}
Obviously, $Q(A,b,\mu,0)$ and $ P(A,b,\mu)$ are the same problem,  and thus $\widehat S(A,b,\mu,0)=S(A,b,\mu)$. Adding the linear perturbation $\ip{v}{x}$ in the problem will be useful  for Section~\ref{Sec:FS} where we leverage our analysis here for more general problems. 

{\color{blue}}

\begin{proposition}[Primal-dual optimality]\label{prop:PD}  The (Fenchel-Rockafellar)  dual of problem~\eqref{eq:Q} reads 
\begin{equation}\label{eq:D}
D(A,b,\mu,v): \quad \min_{y\in\bY}\quad  \frac{\mu}{2}\|y\|^2-\ip{b}{y}+g^*(A^*y+v).
\end{equation} 
For  $(A,b,\mu,v)\in \cL(\bX,\bY)\times \bY\times \R_{++}\times \bX$, the following are equivalent:
\begin{itemize}
    \item[(i)] $x$ solves $Q(A,b,\mu,v)$ and  $y$ solves $D(A,b,\mu,v)$;
    \item[(ii)] $A^*y+v\in \p g(x)$ and $-\mu y=Ax-b$;
    \item[(iii)] $x\in \p g^*(A^*y+v)$ and $-\mu y=Ax-b$.
\end{itemize}
\end{proposition}
\begin{proof}  We apply the duality scheme from \cite[Example 11.41]{RoW98} with $h=\frac{1}{\mu}\|\cdot\|^2, k=g$ and $c=-v$. The only thing we have to ensure is that primal and dual optimal value are equal which is guaranteed as $\dom h=\bY$, cf. \cite[Theorem 11.39]{RoW98}.
% {\color{purple} Note that $x$ is an optimal solution of  $Q(A,b,\mu,v)$ means 
% \begin{equation}\label{eq:Fer1}
%     -\frac{1}{\mu}A^*(Ax-b)-v\in \partial g(x).  
% \end{equation}
% Define $y:=-\frac{1}{\mu}(Ax-b)$, \eqref{eq:Fer1} is equivalent to 
%  $x\in \partial g^*(A^*y+v)$. Hence, to prove the equivalence of (i), (ii), and (iii), we only need to prove that  $y$ is the unique optimal solution of $D(A,b,\mu,v)$ when $x$ solves  $Q(A,b,\mu,v)$. Indeed, note that $
%     g^*(A^*y+v)=\ip{A^*y+v}{x}-g(x)$, as $x\in \partial g^*(A^*y+v)$.  It follows that 
% \begin{eqnarray*}
% \frac{\mu}{2}\|y\|^2-\ip{b}{y}+g^*(A^*y+v)&=&\frac{\mu}{2}\|y\|^2-\ip{b}{y}+\ip{A^*y+v}{x}-g(x)\\
% &=&\frac{\mu}{2}\|y\|^2+\ip{y}{Ax-b}+\ip{v}{x}-g(x)\\
% &=&-\frac{1}{2\mu}\|Ax-b\|^2-g(x)+\ip{v}{x}.
% \end{eqnarray*}
% Due to the strong duality between the primal problem $Q(A,b,\mu,v)$ and the dual problem $D(A,b,\mu,v)$, $y$ is an optimal solution of $D(A,b,\mu,v)$. Since the function in $D(A,b,\mu,v)$ is strongly convex, $y$ the unique solution. The proof is complete.
% }
\end{proof}
Note that an equivalent formulation of the dual problem \eqref{eq:D} reads 
\begin{equation}\label{eq:D'}
 \quad D'(A,b,\mu,v):\quad  \min_{(y,t)\in\bY\times \R}\quad  \frac{\mu}{2}\|y\|^2-\ip{b}{y}+t\st (A^*y+v,t)\in \epi g^*.
\end{equation}
Although the parameter $\mu$ is chosen to be positive for problem~\eqref{eq:P}, problem~\eqref{eq:D'} is well-defined for any $\mu\in \R$.  We will apply the theory laid out for 
the parametric problem OPT$(p)$ in \eqref{p:OPT}  to problem $D'(A,b,\mu,v)$ above with the spaces $\mathbb{U}=\bY\times \R$, $\mathbb{V}=\bX\times\R$, $\mathbb{P}=\mathcal{L}(\bX,\bY)\times \bY\times\R\times \bX$, the parameter $p=(A,b,\mu,v)\in \mathbb{P}$, the variable $u=(y,t)\in \mathbb{U}$, the convex set $\Theta:=\epi g^*\subset \mathbb{V}$, and the functions
\begin{equation}\label{eq:Fts}
    \varphi(u,p):=\frac{\mu}{2}\|y\|^2-\ip{b}{y}+t\quad\mbox{and}\quad  G(u,p):=(A^*y+v,t). 
\end{equation}

To this end, we first give an appropriate expression for the partial constraint nondegeneracy of $D'(A,b,\mu,0)$.

\begin{lemma}[Partial constraint nondegeneracy for dual problem]\label{lem:NDDual}  Let $\bu=(\bar y,\bar t)$ be a solution of $D'(\bar p)$ with fixed $\bp=(\bar A,\bar b,\bar \mu,0)\in \bP$ and $\bar \mu>0$. 
Then the following are equivalent:

\begin{itemize}
\item[(i)] The partial constraint nondegeneracy holds at $(\bar y,\bar t)$ for $D'(\bar A,\bar b,\bar \mu,0)$;
\item[(ii)]$\ker \bar A\cap {\rm par}\, \partial g^*(\bar A^*\bar y)=\{0\}$.
\end{itemize}
 
\end{lemma}

\begin{proof} Set $\bar z:=\bar A^*\bar y$.  In view of \Cref{rem:ND}(b) and with the functions in \eqref{eq:Fts},  we find that constraint nondegeneracy holds  for $D'(\bp)$   at $(\bar y,\bar t)$ if and only if 
\begin{equation}\label{eq:NDD'}
\{0\} = \ker\bar A\times \{0\}\bigcap \spn N_{\epi g^*}(\bar z, \bar t).
\end{equation}
Now, note that, by optimality, we necessarily have $\bar t=g^*(\bar z)$. 
From \Cref{prop:Aux}, we know that 
\[
\bX\times \{0\}\bigcap\spn N_{\epi g^*}(\bar z, \bar t)=\para \p g^*(\bar z)\times \{0\}.
\]
This yields the desired equivalence.
\end{proof}

\noindent
The Lagrangian of $D'(A,b,\mu,v)$ in the sense of \eqref{eq:Lag} reads
\begin{equation}\label{eq:LagD}
    \mathscr{L}(u,p,\lam)=\frac{\mu}{2}\|y\|^2-\ip{b}{y}+t+\ip{\lambda_1}{A^*y+v}+\lam_2 t
\end{equation}
with $u=(y,t)\in \bU=\bY\times \R$, $p=(A,b,\mu,v)\in\bP$, and $\lam=(\lam_1,\lam_2)\in \bV=\bX\times \R$. 
Consequently, stationarity conditions for $D'(A,b,\mu,v)$ in the sense of \eqref{eq:LM} read
\begin{equation}\label{eq:StatD}
    0 = \mu y-b+A\lambda_1,\quad 0= 1+\lam_2, \quad (\lam_1, \lam_2)\in N_{\epi g^*}(A^*y+v,t),
\end{equation}
which certainly tells us that $\lam_2=-1$. 
Combining \eqref{eq:StatD} and \Cref{prop:PD}  yields the following result.

\begin{proposition}\label{prop:Stat} For  $(A,b,\mu,v)\in \cL(\bX,\bY)\times \bY\times \R_{++}\times \bX$, the following are equivalent: 
    \begin{itemize}
       \item[(i)]   $x$ solves $Q(A,b,\mu,v)$ and  $y$ solves $D(A,b,\mu,v)$;
       \item[(ii)] $(y, g^*(A^*y+v))$ solves $D'(A,b,\mu,v)$ and  $(x,-1)$ is a corresponding Lagrange multiplier of problem $D'(A,b,\mu,v)$. 
    \end{itemize}
\end{proposition}
\begin{proof} The proof follows from combining \Cref{prop:PD} and the stationarity conditions in \eqref{eq:StatD} for $D'(A,b,\mu,v)$. Here, it is worth noting that  $(x,-1)\in N_{\epi g^*}(A^*y+v,g^*(A^*y+v))$ if and only if $x\in \p g^*(A^*y+v)$ by \eqref{eq:Sepi}.
\end{proof}

\noindent
To prove Lipschitz stability of the solution map $\widehat S$ around $\bar p=(\bar A,\bar b,\bar \mu,\bar v)$, 
in view of Proposition~\ref{prop:Stat}, it suffices to show that the solution mapping of the stationarity system \eqref{eq:StatD} is locally Lipschitz (in the Lagrange multipliers). By the discussion before Definition~\ref{def:RSR}, the latter occurs when Robinson's strong regularity holds at $(\bu,\bar \lambda)$ for the dual problem $D'(\bp)$.  By Theorem~\ref{thm:RSR}, we need that both the partial constraint nondegeneracy \eqref{eq:ND0} and the generalized strong second-order sufficient condition \eqref{eq:GSSOSC} are satisfied for the system~\eqref{eq:StatD}.  First we show that the generalized SSOSC~\eqref{eq:StatD} is automatic for the dual problem $D'(\bp)$.

\begin{proposition}[Validity of Generalized SSOSC]\label{prop:Val} Let $\bu=(\bar y,\bar t)$ be a solution of $D'(\bar p)$ with $\bar p=(\bar A,\bar b,\bar \mu,0)$ and $\bar \mu>0$. Suppose that $\bar\lam=(\bar \lam_1, -1)$ is the corresponding Lagrange multiplier at $(\bar y,\bar t)$ satisfying \eqref{eq:StatD}. Then the Generalized SSOSC \eqref{eq:GSSOSC} for the dual problem $D^\prime(\bar p)$ holds in the sense that 
\begin{equation}\label{eq:DSSOSC}
   \ip{w}{\nabla^2_{uu}\mathscr{L}(\bu,\bar p, \bar \lambda) w}+\inf\left\{\ip{z}{{\mathcal J} _u G(\bu,\bar{p})w}|\; z\in D^*N_\Theta(G(\bu,\bar p),\bar\lam)({\mathcal J} _u G(\bu,\bar{p})w) \right\} >0
\end{equation} 
 for any $w\in \mathbb{U}\setminus\{0\}$ with all the notations in \eqref{eq:Fts} and \eqref{eq:LagD}.   
\end{proposition}
\begin{proof} 
Note that,  by \eqref{eq:LagD}, we have
 \begin{equation}\label{eq:Hess}
 \nabla^2_{uu}\mathscr{L}(\bu,\bar p, \bar \lam) w=(\bar \mu w_1,0)\quad \forall w=(w_1,w_2)\in \bY\times \R.
 \end{equation}
In particular, $\nabla^2_{uu}\mathscr{L}(\bu,\bar p, \bar \lam)$ is positive semidefinite. Moreover, since the normal cone mapping $N_\Theta$ is  {\em maximally monotone}, we have 
\begin{equation}\label{eq:Min0}
  \inf_{z\in \bX\times \R}\left\{\ip{z}{{\mathcal J} _u G(\bu,\bar{p})w}|\; z\in D^*N_\Theta(G(\bu,\bar p)|\,\bar\lam)({\mathcal J}_u G(\bu,\bar{p})w) \right\}\ge 0, 
\end{equation}
see, e.g., \cite[Theorem~3.1]{PR98}. Therefore,  the left-hand side of \eqref{eq:DSSOSC} is nonnegative for any $w\in \mathbb{U}$. It hence suffices to show that, if it is equal to zero, then $w=0$. To this end, let $w\in \mathbb{U}$ satisfy
 \begin{equation}\label{eq:SSOSC0}
     \ip{w}{\nabla^2_{uu}\mathscr{L}(\bu,\bar p, \bar \lam) w}+\inf_{z\in \bX\times \R}\left\{\ip{z}{{\mathcal J}_u G(\bu,\bar{p})w}|\; z\in D^*N_\Theta(G(\bu,\bar p)|\,\bar\lam)({\mathcal J}_u G(\bu,\bar{p})w) \right\}=0.
 \end{equation}
Applying \eqref{eq:Min0} and \eqref{eq:Hess} to \eqref{eq:SSOSC0} yields  $w_1=0$. And consequently, we  find  that 
\begin{equation}\label{eq:Gra}
   {\mathcal J}_u G(\bu,\bar{p})w=(A^*w_1,w_2)=(0,w_2).
\end{equation}
In particular, \eqref{eq:SSOSC0} reduces to 
\[
\inf_{(z_1,z_2)\in \bX\times \R}\left\{z_2w_2|\; (z_1,z_2)\in D^*N_\Theta(G(\bu,\bar p)|\,\bar\lam)(0,w_2) \right\}=0.
\]
It remains to show that $w_2=0$. To this end, pick  $(z_1,z_2)\in D^*N_\Theta(G(\bu,\bar p)|\,\bar\lam)(0,w_2)$, i.e., 
\[
((z_1,z_2),(0,-w_2))\in N_{{\rm gph}\,N_\Theta}(G(\bu,\bar p),\bar\lam).
\]
By the definition of the normal cone \eqref{def:Nor2}, we find sequences $(v^k,t_k)\to G(\bu,\bar p)=(\bar A^*\by,\bar t)$, $(\lam_1^k,\lam_2^k)\to \bar \lam=(\bar \lam_1,-1)$, $(z_1^k,z_2^k)\to (z_1,z_2)$, and $(w_1^k,w_2^k)\to (0,w_2)$ such that $g^*(v^k)\le t_k$, $(\lam_1^k,\lam_2^k)\in N_\Theta(v^k,t_k)$ and 
\[
\left((z_1^k,z_2^k),(w_1^k,-w_2^k)\right)\in \hat N_{{\rm gph}\,N_\Theta}\left((v^k,t_k),(\lam_1^k,\lam_2^k)\right).
\]
By the expression of  regular normal vectors in \eqref{eq:RegularEps}, for all $\epsilon>0$ there exists $\delta>0$ such that for any $((v,t),(\lam_1,\lam_2))\in \gph N_\Theta\cap\mathbb{B}_\delta((v^k,t_k),(\lam_1^k,\lam_2^k))$:
\begin{equation}\label{eq:Fre}
\ip{z_1^k}{v-v^k}+z_2^k(t-t_k)+   
\ip{w_1^k}{\lam_1-\lam_1^k}-w_2^k(\lam_2-\lam_2^k)\le \epsilon\|  (v,t,\lam_1,\lam_2)-(v^k,t_k,\lam_1^k,\lam_2^k)\|.
\end{equation}
 As  $(\lam_1^k,\lam_2^k)\in N_{\epi g^*}(v^k,t_k)$, we get from \eqref{def:Nor1} that  
\[
\ip{\lam_1^k}{v-v^k}+\lam_2^k(t-t_k)\le 0\quad \mbox{for all}\quad (v,t)\in \epi g^*.
\] 
By choosing $(v,t)=(v^k,g^*(v^k))$ in the above inequality, it follows that $\lam_2^k(g^*(v^k)-t_k)\le 0$. As $\lam_2^k$ is close to $-1$ and $g^*(v^k)\le t_k$, we have $t_k=g^*(v^k)$. 

Let us choose $(\lam_1, \lam_2)$ satisfying $\dfrac{\lam_1}{\lam_2}=\dfrac{\lam_1^k}{\lam_2^k}$ and  such that $(\lam_1,\lam_2)\in \mathbb{B}_\delta(\lam_1^k,\lam_2^k)$. It follows from formula \eqref{eq:Sepi} (which is applicable as $t_k=g^*(v^k)$) that
\[
\dfrac{\lam_1}{-\lam_2}=\dfrac{\lam_1^k}{-\lam_2^k}\in \partial g^*(v^k).
\]
This implies that $(\lam_1,\lam_2)\in N_{\epi g^*}(v^k,t_k)$, i.e., $((v^k,t_k), (\lam_1,\lam_2))\in \gph N_\Theta\cap \mathbb{B}_\delta((v^k,t_k), (\lam^k_1,\lam^k_2))$. We can thus insert $\lambda_1=\lambda_2\frac{\lambda_1^k}{\lambda_2^k}$ and $v=v^k, t=t_k$ in 
\eqref{eq:Fre} to obtain
\[
\frac{\lam_2-\lam_2^k}{\lam_2^k}\ip{w_1^k}{\lam_1^k}-w_2^k(\lam_2-\lam_2^k)\le \epsilon\left[\|\lam_1^k\| \frac{|\lam_2-\lam_2^k|}{|\lam_2^k|}+|\lam_2-\lam_2^k|\right],
\]
which implies that 
\[
-w_2^k(\lam_2-\lam_2^k)\le \epsilon\left[\|\lam_1^k\| \frac{|\lam_2-\lam_2^k|}{|\lam_2^k|}+|\lam_2-\lam_2^k|\right]+\frac{|\lam_2-\lam_2^k|}{|\lam_2^k|}\cdot\left|\ip{w_1^k}{\lam_1^k}\right|.
\]
As $\lam_2$ can be  chosen around $\lam_2^k$, the above inequality implies 
\[
|w_2^k|\le \epsilon\left[\dfrac{\|\lam_1^k\|}{|\lam_2^k|}+1\right]+\dfrac{|\ip{w_1^k}{\lam_1^k}|}{|\lam_2^k|}. 
\]
Since $(\lam_1^k,\lam_2^k)\to (\bar \lam_1,-1)$ and $(w_1^k,w_2^k)\to (0,w_2)$, we get from the above inequality that 
\[
|w_2|\le \epsilon(\|\bar \lam_1\|+1)\quad \mbox{for any}\quad \epsilon>0.
\]
This tells us that $w_2=0$, which verifies the claim at the beginning.
\end{proof}

\noindent
Combining  Lemma~\ref{lem:NDDual}, Proposition~\ref{prop:Stat}, and Proposition~\ref{prop:Val} with Theorem~\ref{thm:RSR}, we obtain the sufficient condition for Lipschitz stability of the solution mapping of parametric problems $P(A, b,  \mu)$ in \eqref{eq:P} and  $Q(A, b,  \mu,v)$ in \eqref{eq:Q} below.

\begin{theorem}\label{thm:Sufi} Let  $\bx$ be a solution of $P(\bar A,\bar b, \bar \mu)$ for $(\bar A,\bar b, \bar \mu)\in \mathcal{L}(\bX,\bY)\times \bY\times\R_{++}$  and suppose that $g^*$ is a $\mathcal{C}^2$-cone reducible function at $\bar z=\frac{1}{\bar \mu}\bar A^*(\bar b-\bar A\bx)$. Assume further that  
\begin{equation}\label{eq:ND}
    \ker \bar A\cap {\rm par}\, \partial g^*(\bar z)=\{0\}.
\end{equation}
Then, the solution mapping $\widehat S$ defined in \eqref{eq:SHat}   
is single-valued and Lipschitz continuous around $(\bar A,\bar b, \bar \mu,0)$ with $\widehat S(\bar A,\bar b, \bar \mu,0)=\bx$. Consequently, the solution mapping $S$ of problem \eqref{eq:P} is also single-valued and Lipschitz continuous around $(\bar A,\bar b, \bar \mu)$ with $S(\bar A,\bar b, \bar \mu)=\bx$. 
\end{theorem} 
\begin{proof} Let $\by$ be the unique solution of the dual problem $D(\bar A,\bar b,\bar\mu,0)$ in \eqref{eq:D}. Note from Proposition~\ref{prop:PD} that $\bar z=A^*\by$. By Lemma~\ref{lem:NDDual}, the partial constraint nondegeneracy holds at $(\by, g^*(\bar z))$ for the dual problem $D'(\bar A,\bar b,\bar \mu,0)$ in \eqref{eq:D'}. By Proposition~\ref{prop:Val}, the generalized SSOSC for the dual problem $D^\prime(\bar A,\bar b,\bar \mu,0)$ is also satisfied. Moreover, since $g^*$ is $\mathcal{C}^2$-cone  reducible at $\bar z$, its epigraph $\epi g^*$ is $\mathcal{C}^2$-cone  reducible at $G(\bar u, \bar p)$ from \eqref{eq:Fts} with $\bar u=(\bar y, g^*(\bar z))$ and $\bar p=(\bar A,\bar b, \bar \mu,0)$. We are in the position of applying Theorem~\ref{thm:RSR} to the dual problem \eqref{eq:D'}. By the discussion after Definition~\ref{def:RSR}, 
the Lagrange multiplier mapping $\Lam$ of $D^\prime(A, b,\mu,v)$ must have a single-valued and Lipschitz continuous localization around $(\bar A,\bar b,\bar \mu,0)$ for  $\bx$. By Proposition~\ref{prop:Stat}, $\Lam(A, b,\mu,v)=\widehat S(A,b,\mu,v)\times\{1\}$. Thus the solution mapping $\widehat S$ has a single-valued and Lipschitz continuous localization around $(\bar A,\bar b,\bar \mu,0)$ for $\bx$. 
 The desired statements now follow from \Cref{cor:SingleVal} and the fact that $S(A,b,\mu)=\widehat S(A,b,\mu,0)$.
\end{proof}

\begin{remark}[The class of $\mathcal{C}^2$-cone reducible conjugates]\label{rem:Ex} {\rm  Our main assumption on the regularizer $g$ in Theorem~\ref{thm:Sufi} to have a conjugate $g^*$ that  is $\mathcal{C}^2$-cone reducible at $\bar z$ appears to be restrictive at  first glance.
% as we use Theorem~\ref{thm:RSR} for the dual problem $D'(\cdot)$ in \eqref{eq:D'}. 
However, it turns out that many important convex regularizers $g$ satisfy this condition. We furnish evidence for this by the following list of (important) examples:

\begin{itemize}
    \item[(a)] {\em (Convex piecewise linear functions)} The proper convex function $g:\bX\to\R\cup\{\infty\}$ is called  convex {\em piecewise linear} on its domain if   $\dom g$ can be represented as the union of finitely many polyhedral sets, relative to each of which $g(x)$ is given by an expression $\ip{a}{x}+b$ for some $a\in \bX$ and $b\in \R$; see, e.g., \cite[Definition~2.47]{RoW98}. The Fenchel conjugate of $g$ is also a piecewise linear function according to \cite[Theorem~11.14]{RoW98}. Thus, the epigraph of $g^*$ is a polyhedra by \cite[Theorem~2.49]{RoW98}.   It follows from \cite[Example~3.139]{BS00} that $\epi g^*$ is $\mathcal{C}^2$-cone reducible. This tells us that the conjugate of any proper convex piecewise linear function is $\mathcal{C}^2$-cone reducible.

    \item[(b)] {\em (Support functions of $\mathcal C^2$-cone reducible set)} Let  $g$ be the {\em support function} of a closed, convex, and $\mathcal{C}^2$-cone reducible set $C\subset \bX$, i.e., 
\begin{equation}\label{eq:support}
g(x)=\sigma_C(x):=\sup\set{\ip{v}{x}}{v\in C}.
\end{equation}
Note that $g^*=\delta_C$, the indicator function of $C$. Hence $\epi g^*=C\times \R_+$. Obviously, $\R_+$  is polyhedral, hence $\mathcal{C}^2$-cone reducible \cite[Example~3.139]{BS00}. Thus,  the Cartesian product $C\times \R_+$ is $\mathcal{C}^2$-cone reducible. It follows that $g^*$ is $\mathcal{C}^2$-cone reducible. 

This class of support functions is broad. It includes norm functions that are widely used as regularizers, see the two following examples.

\item[(c)] {\em ($\ell_1/\ell_2$ norm)} Let 
$g$  be  the $\ell_1/\ell_2$ norm in $\R^n$ given  by
\begin{equation}\label{eq:l12}
     g(x)=\|x\|_{1,2}:=\sum_{J\in \mathcal{G}}\|x_J\|\quad \mbox{for any}\quad x=(x_J)_{J\in \mathcal{G}}\in \R^n,
\end{equation}
where $\mathcal{G}$ is a partition of $\{1,2,\ldots,n\}$ with $m$ different groups $J\in \mathcal{G}$ and $\|x_J\|$ is the Euclidean norm at $x_J\in\R^{|J|}$ for any $J\in \mathcal{G}$.  We note that this covers the $\ell_1$ norm by the partition $\set{\{i\}}{i=1,\dots, n}$.

 Now, define the set $C\subset \R^n$ by 
\[
C:=\prod_{J\in \mathcal{G}}\mathbb{B}_J,
\]
\noindent
where $\mathbb{B}_J$ is the  (Euclidean) unit ball in $\R^{|J|}$.  Then it is easy to see that $g=\sig_C$. Note that $\mathbb{B}_J=\{v\in \R^{|J|}|\,h(v):=1-\|v\|^2\in \R_+\}$ is $\mathcal{C}^2$-cone reducible at any $v$ with $\|v\|=1$, as $h(v)=0$ and ${\mathcal J} h(v):\R^{|J|}\to \R, x\mapsto -2v^Tx$ is surjective. As discussed after Definition~\ref{def:Redu}, $\mathbb{B}_J$ is also  $\mathcal{C}^2$-cone reducible at any $\bv\in {\rm int}\, \mathbb{B}_J$. Thus, it is $\mathcal{C}^2$-cone reducible at any $\bv\in \mathbb{B}_J$. It follows that the Cartesian product $C$ defined above is  $\mathcal{C}^2$-cone reducible.
Consequently, $g^*=\delta_C$ is $\mathcal{C}^2$-cone reducible. 
% Hence the Fenchel conjugate of  $\ell_1/\ell_2$ norm (and the $\ell_1$ norm in particular) is  $\mathcal{C}^2$-cone reducible.

\item[]

\item[(d)]{{\em (Nuclear norm)}} Another important regularizer is the {\em nuclear norm} $\|X\|_*$, the sum of all {\em singular values} of $X\in \R^{m\times n}$ ($m\le n$). Its Fenchel conjugate is also $\mathcal{C}^2$-cone reducible; see, e.g., \cite[Proposition~3.2 and Remark~3.4]{CDZ17}. This fact can be explained directly and differently as follows. The nuclear norm can be written as the support function $\|X\|_*=\sigma_{\mathbb B}(X)$, where $\mathbb{B}=\{Z\in \R^{m\times n}|\; \|Z\|_2\le 1\}$ is the {\em spectral unit ball} in $\R^{m\times n}$ with $\|\cdot\|_2$ being the spectral norm. We show next that $\mathbb{B}$ is $\mathcal{C}^2$-cone reducible at any $\bar Z$ on the boundary of $ \mathbb{B}$, i.e., $\|\bar Z\|_2=1$.  Note that the spectral unit ball is represented by 
\begin{equation}\label{eq:Spec}
\mathbb{B}=\left\{Z\in \R^{m\times n}|\, G(Z):=\mathbb{I}_n-Z^TZ\in \mathbb{S}^{n}_+\right\},
\end{equation}
where $\mathbb{I}_{n}$ is the $n\times n$ identity matrix and   $\mathbb{S}^{n}_+$ is the set of all positive semidefinite $n\times n$ matrices. As the positive semidefinite cone $\mathbb{S}^n_+$ is $\mathcal{C}^2$-cone reducible \cite[Example~3.140]{BS00}, it suffices to show that $\bar Z$ is nondegenerate for $G$ with respect to $\mathbb{S}^{n}_+$  according to \cite[Proposition~3.2]{Sh03}, which means 
\begin{equation}\label{eq:BND}
    \rge{\mathcal J} G(\bar Z)+\lin T_{\mathbb{S}^{n}_+}(G(\bar Z))=\mathbb{S}^{n},
\end{equation}
where $\mathbb{S}^{n}$ is the space of all $n\times n$ symmetric matrices. Suppose the Singular Value Decomposition (SVD) of $\bar Z$ is  $\bar Z=\bar U \Sigma \bar V^T$, where $\bar U\in \R^{m\times m}$ and $\bar V\in \R^{n\times n}$ are orthogonal matrices and $\Sigma=\begin{pmatrix}{\rm Diag}\,(\sigma_1, \ldots, \sigma_m)&0\end{pmatrix}$ has two matrix blocks: the first block of $m\times m$ is the diagonal matrix  containing all singular values $\sigma_1\ge \sigma_2\ldots\ge \sigma_m$ of $\bar Z$  and the second block is just zero $m \times (n-m)$ matrix. Note that 
\begin{equation}\label{eq:GraG}
{\mathcal J} G(\bar Z)X=-\bar Z^T X- X^T\bar Z=-\bar V\Sigma^T \bar U^TX-X^T\bar U\Sigma \bar V^T \quad \mbox{for any}\quad X\in \R^{m\times n}.
\end{equation}
Since $\|\bar Z\|_2=1$, we have $\sigma_1=1$. Let $p$ be the multiplicity of $\sigma_1$ in $\{\sigma_1, \ldots, \sigma_m\}$, i.e., $1=\sigma_p>\sigma_{p+1}$ (with $\sigma_{m+1} := 0$). Choose $X=\bar U\begin{pmatrix}A&0\\0&0\end{pmatrix} \bar V^T$ for any $A\in \mathbb{S}^p$, we derive from \eqref{eq:GraG} that 
\[
{\mathcal J} G(\bar Z)X=-\bar V\Sigma^T\begin{pmatrix}A&0\\0&0\end{pmatrix} \bar V^T-\bar V \begin{pmatrix}A&0\\0&0\end{pmatrix}\Sigma \bar V^T=
-2\bar V\begin{pmatrix}A&0\\0&0\end{pmatrix}\bar V^T.
\]
It follows that 
\begin{equation}\label{eq:ImG}
    \rge\, {\mathcal J} G(\bar Z)\supset\left\{\bar V\begin{pmatrix}A&0\\0&0\end{pmatrix}\bar V^T|\; A\in \mathbb{S}^p\right\}.
\end{equation}
Moreover, observe further that 
\begin{eqnarray*}\begin{array}{ll}
T_{\mathbb{S}^{n}_+}(G(\bar Z))&=\; \; T_{\mathbb{S}^{n}_+}(\bar V(\mathbb{I}_n-\Sigma^T\Sigma)\bar V^T)\\
&=\; \; T_{\mathbb{S}^{n}_+}\left(\bar V\begin{pmatrix}0_{p\times p}&0&0\\0&{\rm Diag}\,(1-\sigma^2_{p+1}, \ldots, 1-\sigma^2_m)&0\\0&0&\mathbb{I}_{n-m}
\end{pmatrix}\bar V^T\right).
\end{array}
\end{eqnarray*}
By the well-known formula of tangent cone of the positive semidefinite cone \cite[Example~2.65]{BS00}, we have 
\[
T_{\mathbb{S}^{n}_+}(G(\bar Z))=\left\{\bar V\begin{pmatrix}A&B\\B^T&C\end{pmatrix}\bar V^T|\; A\in \mathbb{S}_+^p, B\in \R^{p\times (n-p)}, C\in  \mathbb{S}^{n-p}\right\}. 
\]
It follows that 
\[
\lin\, T_{\mathbb{S}^{n}_+}(G(\bar Z))=\left\{\bar V\begin{pmatrix}0&B\\B^T&C\end{pmatrix}\bar V^T|\; B\in \R^{p\times (n-p)}, C\in  \mathbb{S}^{n-p}\right\},
\]
which together with \eqref{eq:ImG} verifies \eqref{eq:BND}. Thus, $\bar Z$ is nondegenerate for $G$ with respect to $\mathbb{S}^n_+$. The preimage $\mathbb{B}=G^{-1}(\mathbb{S}^n_+)$ is $\mathcal{C}^2$-cone reducible at $\bar Z$ due to  \cite[Proposition~3.2]{Sh03}. As $\mathbb{B}$ is always $\mathcal{C}^2$-cone reducible at point in is interior, it is $\mathcal{C}^2$-cone reducible at any $\bar Z\in \mathbb{B}$. This explains the fact that the Fenchel conjugate of the nuclear norm is  $\mathcal{C}^2$-cone reducible. There are many more spectral functions whose Fenchel conjugates are also $\mathcal{C}^2$-cone reducible \cite[Proposition~3.2]{CDZ17}.
\end{itemize}
 \hfill$\triangle$
}
\end{remark}

\noindent
 We now show that qualification condition \eqref{eq:ND} is necessary for the solution map $S(\bar A,\cdot, \bar \mu)$ in \eqref{eq:Sp} to be single-valued. 

% Next, we provide a necessary condition for single-valuedness of the solution mapping $S$ in \eqref{eq:Sp}. This condition becomes sufficient when the function $g^*$ is $\mathcal{C}^2$-cone reducible in our later Theorem~\ref{thm:Full}. 

\begin{theorem}[Necessity of single-valuedness of solution mapping $S(\bar A,\cdot, \bar \mu)$]\label{thm:Nes} Suppose that  the solution mapping $S(\bar A, b, \bar\mu)$ in \eqref{eq:Sp} is single-valued around $\bar b$ with  $\bx=S(\bar A,\bar b, \bar \mu)$. Then, with $\bar z=\frac{1}{\bar \mu}\bar A^*(\bar b-\bar A\bar x)$,    condition~\eqref{eq:ND} is satisfied.
    
\end{theorem}
\begin{proof}Suppose that  the solution mapping $S(\bar A,b,\bar \mu)$ in \eqref{eq:Sp} is single-valued around $\bar b$ with $\bx=S(\bar A,\bar b, \bar \mu)$. By first-order optimality conditions, we have $\bar z\in \partial g(\bx)$, i.e., $\bar x\in \partial g^*(\bar z)$. Since $\partial g^*(\bar z)$ is a closed convex set, we have 
\begin{equation}\label{eq:Ri}
\cl[{\rm ri}\, \partial g^*(\bar z)]=\partial g^*(\bar z),
\end{equation}
see, e.g., \cite[Theorem~6.3]{R70}.  It follows that there exists a sequence $\{x_k\in  {\rm ri}\, \partial g^*(\bar z)\}\to \bar x$.  Therefore, we find that 
\[
b_k:=\bar b+\bar A(x_k-\bx)\to \bar b. 
\]
Observe that 
\[
\frac{1}{\bar \mu}\bar A^*(b_k-\bar Ax_k)=\frac{1}{\bar \mu}\bar A^*(\bar b-\bar A\bar x)=\bar z\in \partial g(x_k). 
\]
By first-order optimality, we thus infer that  $x_k\in S(\bar A,b_k,\bar \mu)$ is the (by assumption) unique solution of the problem $P(\bar A, b_k,\bar \mu)$ for all $k\in\bN$ sufficiently large. By \cite[Theorem~3.1]{FNP24}, which provides a full characterization of solution uniqueness for convex optimization problems,  we  obtain that
\begin{equation}\label{Uniq}
    \ker\bar A\cap \cone(\partial g^*(\bar z)-x_k)=\{0\}
\end{equation}
 for all $k\in \bN$ sufficiently large.
As $x_k\in \ri \partial g^*(\bar z)$, we have $\cone\,(\partial g^*(\bar z)-x_k)=\para \p g^*(\bar z)$ (cf.~\eqref{eq:ParC}).
This together with \eqref{Uniq} verifies the condition \eqref{eq:ND}. 
\end{proof}

We are now in a position to establish the  the main result of this paper, which shows that condition \eqref{eq:ND} is not only sufficient for Lipschitz stability of the solution mapping \eqref{eq:Sp} around $(\bar A,\bar b,\bar \mu)$, but also necessary provided that the function $g^*$ is $\mathcal{C}^2$-cone reducible. Moreover, we show that the solution mapping $S$ is single-valued and Lipschitz continuous around the point in question if and only if it is single-valued, i.e., the Lipschitz stability is automatic whenever the mapping is single-valued. This equivalence is interesting, as the Lipschitz stability is usually not a free property for single-valued mappings in general.

\begin{theorem}[Full characterization of Lipschitz stability of solution mapping $S$]\label{thm:Full} Suppose that  $\bar x$ is an optimal solution of problem $P(\bar A,\bar b,\bar \mu)$ and that the function $g^*$ is $\mathcal{C}^2$-cone reducible at $\bar z=\frac{1}{\bar \mu}\bar A^*(\bar b-\bar A\bar x)$. Then the following  are equivalent: 
\begin{itemize}
    \item[{\bf(i)}] Condition \eqref{eq:ND} is satisfied. 
    \item[{\bf(ii)}] The solution mapping $S(A,b,\mu)$ in \eqref{eq:Sp} of problem~\eqref{eq:P} is single-valued and Lipschitz continuous around $(\bar A,\bar b,\bar \mu)$ with $\bx=S(\bar A,\bar b,\bar \mu)$.
    
    \item[{\bf(iii)}] The solution mapping $S(\bar A,\cdot,\cdot)$ is single-valued and Lipschitz continuous around $(\bar b,\bar \mu)$ with $\bx=S(\bar A,\bar b,\bar \mu)$.
    \item[{\bf(iv)}] The solution mapping $S(\bar A,\cdot,\bar \mu)$ is single-valued and Lipschitz continuous around $\bar b$ with $\bx=S(\bar A,\bar b,\bar \mu)$.
    
    \item[{\bf(v)}]  The solution mapping $S$ in \eqref{eq:Sp} is single-valued around $(\bar A,\bar b,\bar \mu)$ with $\bx=S(\bar A,\bar b,\bar \mu)$.

    \item[{\bf(vi)}] The solution mapping $S(\bar A,\cdot, \cdot)$ is single-valued  around $(\bar b,\bar \mu)$ with $\bx=S(\bar A,\bar b,\bar \mu)$.

     \item[{\bf(vii)}] The solution mapping $S(\bar A,\cdot,\bar \mu)$ is single-valued  around $ \bar b$ with $\bx=S(\bar A,\bar b,\bar \mu)$.
\end{itemize}
\end{theorem}
\begin{proof} The implication [{\bf (i)}$\Rightarrow${\bf (ii)}] follows from Theorem~\ref{thm:Sufi}. Implications [{\bf (ii)}$\Rightarrow${\bf (iii)}$\Rightarrow${\bf (iv)}$\Rightarrow${\bf (vii)}] are trivial. We also have  [{\bf (ii)}$\Rightarrow${\bf (v)}$\Rightarrow${\bf (vi)}$\Rightarrow${\bf (vii)}]. Moreover, [{\bf (vii)}$\Rightarrow${\bf (i)}] is obtained from Theorem~\ref{thm:Nes}. 
\end{proof}

Theorem~\ref{thm:Full} generalizes the recent results in \cite[Theorem~3.12 and Theorem~4.7]{N24}, which establish the equivalence between {\bf (iii)} and {\bf (vi)} when regularizers $g$ include the $\ell_1/\ell_2$ norm and the nuclear norm via a different approach. Note from our Remark~\ref{rem:Ex}, conjugates of the latter two regularizers are $\mathcal{C}^2$-cone reducible. Moreover, our condition \eqref{eq:ND} recovers the characterizations of Lipschitz stability for the solution mapping in \eqref{eq:Sp} used in \cite{N24} for these two cases and also the one in \cite[Assumption~4.3]{BBH23} when $g$ is the $\ell_1$-norm, which appeared in \cite{MY12,T13} as a sufficient condition for the Lasso problem. The implication [{\bf (vii)}$\Longrightarrow${\bf (iv)}] in this theorem is also stronger than \cite[Theorem~3.3]{EDA24}, which shows that if $S(\bar A,\cdot,\bar \mu)$ is single-valued around $\bar b$, then it is continuous when $g$ is a convex piecewise linear function that is $\mathcal{C}^2$-cone reducible;  see, e.g., Remark~\ref{rem:Ex}. Stability of the solution mapping $S$ with respect to variable $b$ seems to be important in the proof of our Theorem~\ref{thm:Nes} and Theorem~\ref{thm:Full}. By fixing the parameter $b=\bar b$, our condition \eqref{eq:ND} is also sufficient for the Lipschitz stability of solution mappings $S(\cdot, \bar b,\bar \mu)$, $S(\cdot, \bar b,\cdot)$, and $S(\bar A, \bar b, \cdot)$ around the corresponding points due to the implication [{\bf (i)}$\Rightarrow${\bf (ii)}] in Theorem~\ref{thm:Full}. Whether it is necessary is an open question.

 We close out this section by showing that the Lipschitz stability of the solution mapping of problem \eqref{eq:P} is  automatic  when $\bar x$ is the unique solution of problem $P(\bar A,\bar b,\bar \mu)$ and the so-called {\em Dual Strict Complementarity Condition} is satisfied and  granted that  $g^*$ is $\mathcal{C}^2$-cone reducible.

\begin{theorem} Suppose that  $\bar x$ is an optimal solution of problem $P(\bar A,\bar b,\bar \mu)$ and that the function $g^*$ is $\mathcal{C}^2$-cone reducible at $\bz=\frac{1}{\bar \mu}\bar A^*(\bar b-\bar A\bar x)$. Moreover, suppose that the following Dual Strict Complementarity Condition holds
\begin{equation}\label{eq:DSCC}
\bar x\in {\rm ri}\,\partial g^*(\bar z). 
\end{equation}
Then the solution mapping $S$ in \eqref{eq:Sp} is single-valued and Lipschitz continuous around $(\bar A,\bar b,\bar \mu)$ with $S(\bar A,\bar b,\bar \mu)=\bar x$ if and only if $\bar x$ is the unique solution of problem $P(\bar A,\bar b,\bar \mu)$.
\end{theorem}
\begin{proof} The direction ``$\Rightarrow$'' is trivial. To prove the opposite implication, we suppose that $\bar x$ is the unique solution of problem $P(\bar A,\bar b,\bar \mu)$ and that the condition \eqref{eq:DSCC} is satisfied. By \cite[Theorem~3.1]{FNP24}, we have
\begin{equation}\label{eq:Uni}
    \ker\, \bar A\cap \cone(\partial g^*(\bar z)-\bar x)=\{0\}\quad \mbox{with}\quad \bar z=-\frac{1}{\bar \mu}\bar A^*(\bar A\bar x-\bar b).
\end{equation}
As $\bar x\in {\rm ri}\,\partial g^*(\bar z)$ by \eqref{eq:DSCC}, we obtain from  \eqref{eq:ParC} that 
\[
\cone\,(\partial g^*(\bar z)-\bar x)={\rm par}\, \partial g^*(\bar z).
\]
This together \eqref{eq:Uni} tells us that condition~\eqref{eq:ND} is satisfied. By Theorem~\ref{thm:Full}, the solution mapping $S$ is single-valued and Lipschitz continuous around $(\bar A,\bar b,\bar \mu)$ with $S(\bar A,\bar b,\bar \mu)=\bar x$.
\end{proof}

The above result  resembles   \cite[Corollary~4.8]{N24} when $g$ is the nuclear norm and the condition \eqref{eq:DSCC} is replaced by the {\em Strict Complementarity Condition}:
\begin{equation}\label{eq:SCC}
    \bar v\in \ri\, \partial g(\bx).
\end{equation}
For the case of nuclear norm, it is shown in \cite[Remark~3.7]{FNP24} that \eqref{eq:DSCC} and \eqref{eq:SCC} are equivalent. But at this stage, we do not know their correlation for more general  $\mathcal{C}^2$-cone reducible functions $g^*$.

%\begin{example} The following example shows that \Cref{thm:Full} may fail in general. To this end consider the problem 

 %\begin{equation}\label{eq:ExFail} 
 %\frac{1}{2}
 %\end{equation}

%\end{example}

\section{Full stability of convex additive-composite problems}\label{Sec:FS}

In this section, we study the following  optimization problem 
\begin{equation}\label{eq:OP}
 \Tilde P(\bar p)\qquad    \min_{x\in \bX}\quad  f(x,\bar p)+g(x),
\end{equation}
where $\bX$ and $\bP$ are two Euclidean spaces,  $f:\bX\times \mathbb{P}\to \R\cup\{+\infty\}$,   and $g:\bX\to \R\cup\{+\infty\}$ is  closed, proper,  convex. Suppose that $\bar x\in \dom f(\cdot, \bar p)\cap \dom g$ is an optimal solution of problem~\eqref{eq:OP}.  Throughout this section, we assume that 
\begin{itemize}
    \item[(A)] $f(\cdot, \bar p)$ is a convex function;
    \item[(B)] $f(\cdot,\cdot)$ is twice continuously differentiable around   $(\bar x,\bar p)$.
\end{itemize}
Define the function $\varphi:\bX\times \bP\to\R\cup\{+\infty\}$ by
\begin{equation}\label{eq:vp}
\varphi(x,p):=f(x,p)+g(x)\quad \mbox{for all}\quad (x,p)\in \bX\times \bP.
\end{equation}
The two-parametric perturbation of problem \eqref{eq:OP} is constructed by 
\begin{equation}\label{eq:OPpv}
    \min_{x\in \bX}\quad  \varphi(x,p)-\ip{v}{x},
\end{equation}
with {\em basic} perturbation $p\in \mathbb{P}$ and tilt parameter $v\in \bX$. This certainly covers the problem \eqref{eq:Q}. 

Let us recall the definition of {\em full stability} introduced by Levy, Poliquin, and Rockafellar \cite{LPR00}, which is a far-reaching extension of the notion of  tilt stability by Poliquin and Rockafellar \cite{PR98}. 

\begin{definition}[Full stability and tilt stability]\label{TS} The point $\bx$ is called a {\em fully stable} optimal solution of  problem \eqref{eq:OP} if there exists $\gam>0$ such that the solution map
\begin{equation}\label{Mg}
M_\gam(v,p):={\rm argmin}\; \{\varphi(x,p)-\ip{v}{x}|\; x\in \mathbb{B}_\gam(\bx)\}\quad \mbox{for}\quad (v,p)\in \bX\times \bP
\end{equation}
and the mapping 
\begin{equation}\label{def:mg}
m_\gam(v,p):=\inf\;\{\varphi(x,p)-\ip{v}{x}|\; x\in \mathbb{B}_\gam(\bx)\}
\end{equation}
are single-valued and Lipschitz continuous on some neighborhood of $(0, \bar p)\in \bX\times \bP$ with $M_\gam(0,\bar p)=\bx$.

The point $\bar x$ is called a {\em tilt stable} optimal solution  of problem \eqref{eq:OP} if the mapping $M_\gam(\cdot,\bar p)$  is single-valued and Lipschitz continuous on some neighborhood of $0\in \bX$ with $M_\gam(0,\bar p)=\bx$.

\end{definition}
As the function $f$ is twice continuously differentiable around $(\bar x,\bar p)$ { and $g$ is  closed, proper, convex}, the function $\varphi$ is {\em continuously prox-regular} in $x$ at $\bx$ for $0$, with compatible parameterization by $p$ at $\bar p$ in the sense of \cite[Definition~2.1]{LPR00}. Moreover, we observe that the so-called {\em basic constraint qualification}
%; see also  \cite[Proposition~10.16]{RoW98}
\begin{equation}\label{eq:BCQ}
(0,q)\in \partial^\infty \varphi(\bar x,\bar p)\quad   \Longrightarrow\quad q=0
\end{equation}
 is also satisfied, where $\partial^\infty \varphi(\bar x,\bar p)$ is the singular subdifferential of the function $\varphi$ at $(\bar x,\bp)$ defined in \eqref{eq:sinS}.   This is due to the sum rule of the singular subdifferential \cite[Proposition~1.107]{M06} (see also \cite[Corollary~10.9]{RoW98}),  which gives
\[
\partial^\infty \varphi(\bar x,\bar p)=\partial^\infty g(\bar x)\times \{0\}
\]
using assumption (B). According to \cite[Proposition~3.5]{LPR00}, when the above  basic constraint qualification is satisfied,  the condition in the definition of full stability that $m_\gam(\cdot,\cdot)$ is Lipschitz continuous on some neighborhood of $(0, \bar p)\in \bX\times \bP$ is automatically met, provided that $M_\gam(0,\bar p)=\{\bx\}$. It is also worth mentioning that the latter condition $M_\gam(0,\bar p)=\bx$ means that $\bx$ is the unique (global) optimal solution of problem \eqref{eq:OP}, since both functions $f(\cdot, \bar p)$ and $g(\cdot)$ are convex in our framework.

Let us recall here \cite[Theorem~2.3]{LPR00}, \cite[Theorem~4.4]{MNR15}, and  \cite[Theorem~1.3]{PR98}, which provides  characterizations of full stability and tilt stability via second-order analysis.

\begin{theorem}[Characterizations of full stability]\label{thm:FS} Suppose that $\bx$ is an optimal solution of problem \eqref{eq:OP}. Then $\bx $ is a fully stable optimal solution of problem \eqref{eq:OP} if and only if the following two conditions hold:
\begin{itemize}
\item[(a)] $\inf\,\{\ip{z}{w}|\; (z,q)\in D^*(\partial_x \varphi)(\bar x,\bar p|\, 0)(w)\}>0$ for all $w\in \bX\setminus\{0\}$. 

\item[(b)] $(0,q)\in D^*(\partial_x \varphi)(\bar x,\bar p|\, 0)(0)$ $\Longrightarrow$ $q=0$.
\end{itemize}
Moreover, $\bx$ is a tilt stable optimal solution of problem \eqref{eq:OP} if and only if
\begin{equation}\label{eq:Tilt}
\inf\,\{\ip{z}{w}|\; z\in D^*\partial\varphi_{\bar p} (\bar x|\,0)(w)\}>0\quad\mbox{for all}\quad w\in \bX\setminus\{0\},
\end{equation}
where $\varphi_{\bar p}(\cdot):=\varphi(\cdot, \bar p)$. 
\end{theorem}

It is worth mentioning that in the original result of Levy, Poliquin, and Rockafellar \cite[Theorem~2.4]{LPR00},  the condition in the above part (a) does not include the ``infimum''. This equivalent ``infimum'' form  coming from \cite[Theorem~4.4]{MNR15}  is  useful for us in this paper at some point.

Note also that $\partial_x \varphi(x,p)=\nabla_x f(x,p)+\partial g(x)$. It follows from the coderivative sum rule  \cite[Theorem~1.62]{M06}  that 
\begin{equation}\label{eq:Cod}
D^*(\partial_x \varphi)(\bar x,\bar p|\, 0)(w)=(\nabla^2_{xx} f(\bx,\bar p)^*w,\nabla_{xp} f(\bx,\bar p)^*w)+D^*\partial g(\bar x|\, -\nabla_x f(\bx,\bar p))(w)\times \{0\}.
\end{equation}
This tells us that condition~(b) in Theorem~\ref{thm:FS} is always true in our setting. Moreover, condition~(a) in  Theorem~\ref{thm:FS} turns to 
\begin{equation}\label{Tilt}
    \ip{\nabla^2_{xx} f(\bx,\bar p)w}{w}+\inf\{\ip{z}{w}|\; z\in D^*\partial g(\bx|\,
-\nabla_x f(\bx,\bar p))(w)\}>0 \quad \mbox{for all}\quad w\in \bX\setminus\{0\}, 
\end{equation}
which is also  condition \eqref{eq:Tilt} again due to the 
the  sum rule of the coderivative \cite[Theorem~1.62]{M06}: 
\[
D^*\partial_x \varphi_{\bar p}(\bar x|\,0)(w)=\nabla^2_{xx} f(\bar x,\bar p)^*w+D^*\partial g(\bar x|\, -\nabla_x f(\bx,\bar p))(w).
\]
Hence, both full stability and tilt stability are equivalent for our problem~\eqref{eq:OP} at the optimal solution $\bar x$ and they are characterized by condition~\eqref{Tilt}.

As $f(\cdot,\bar p)$ is a convex function, $\nabla^2_{xx} f(\bx,\bar p)$ is a positive semidefinite operator. Its {\em square root operator} $\Bar A:=(\nabla^2_{xx} f(\bx,\bar p))^{\frac{1}{2}}\in \mathcal{L}(\bX,\bX)$ with $\bar A^*\bar A=\nabla^2_{xx} f(\bx,\bar p)$ is unique and well-defined. Condition \eqref{Tilt} is equivalent to 
\begin{equation}\label{eq:Ati}
\|\bar A w\|^2+\inf\{\ip{z}{w}|\; z\in D^*\partial \tilde g(\bx|\,
0)(w)\}>0 \quad \mbox{for all}\quad w\in \bX\setminus\{0\}
\end{equation}
with $\tilde g(x):= g(x)+\ip{\nabla_x f(\bx,\bp)}{x-\bx}$ for  any $x\in \bX$. It follows from Theorem~\ref{thm:FS} that the above condition is satisfied if and only if 
$\bx$ is a tilt stable optimal solution of the following problem 
\begin{equation}\label{eq:CP2}
    \min_{x\in \bX}\quad \frac{1}{2}\|\Bar Ax-\bar b\|^2+\tilde g(x)\quad \mbox{with}\quad \bar b:=\Bar A\bx.
\end{equation}
This optimization problem is in the format of \eqref{eq:P}. The above observation allows us to establish a simple necessary condition of the tilt stability as below.

\begin{theorem}[Necessary condition for tilt stability] \label{thm:NeTi}Suppose that $\bx$ is a tilt stable optimal solution of problem \eqref{eq:OP}. Then we have 
\begin{equation}\label{con:ND2}
    \ker \nabla_{xx}^2 f(\bx,\bar p)\cap {\rm par}\, \partial g^*(\bar z)=\{0\}\quad\mbox{with}\quad \bar z:=-\nabla_x f(\bx,\bar p).  
\end{equation}

\end{theorem}
\begin{proof} Suppose that $\bx$ is a tilt stable optimal solution of problem \eqref{eq:OP}. As discussed above, it is also a tilt stable optimal solution of problem \eqref{eq:CP2}.  In particular, by first-order optimality conditions, $\bar z\in \p g(\bar x)$,  i.e., $\bar x\in \p g^*(\bar z)$, thus  \eqref{con:ND2} is well-defined. Therefore, 
% By equality \eqref{eq:Ri}{\color{red} (I am unsure if this is the right reference.)}, 
there exists a sequence $\{x_k \in\ri\partial g^*(\bar z)\}\to \bar x$. Define 
\[
v_k:=\nabla^2_{xx} f(\bx,\bar p)(x_k-\bx) \to 0
\]
and consider the  linearly perturbed problem 
\begin{equation}\label{eq:CP3}
 \min_{x\in \bX}\quad \psi_k(x):=\frac{1}{2}\|\Bar Ax-\bar b\|^2+\tilde g(x) -\ip{v_k}{x}
\end{equation}
As $\bar z\in \partial g(x_k)$ and $\Bar A^*\Bar A=\nabla^2_{xx} f(\bx,\bar p)$, we have  
\[
\partial \psi_k(x_k)=\Bar A^*(\Bar Ax_k-\Bar A\bx)+\partial g(x_k)-\bar z-v_k=\nabla^2_{xx} f(\bx,\bar p)(x_k-\bx)+\partial g(x_k)-\bar z-v_k\ni 0. 
\]
Hence,  $x_k$ is an optimal solution of problem \eqref{eq:CP3}. Since $v_k\to 0$ and $\bx$ is a tilt stable optimal solution of problem \eqref{eq:CP2}, $x_k$ is the unique solution of problem \eqref{eq:CP3} (for all $k\in\bN$ sufficiently large). Define $\hat g_k(x):= \tilde g(x) -\ip{v_k}{x}$ and obtain from  \cite[Theorem~3.1]{FNP24} that 
\begin{equation}\label{eq:Hori}
\ker \Bar A\cap \cone(\partial \hat g_k^*(-v_k)-x_k)=\{0\}, 
\end{equation}
which is indeed the characterization for solution uniqueness of problem~\eqref{eq:CP3}. Note that  
\[
\hat g_k(x)= g(x)-\ip{\bar z}{x-\bar x}-\ip{v_k}{x}.\]
We have $\partial \hat g_k(x)=\partial g(x)-\bar z-v_k$. It follows that $\partial \hat g_k^*(-v_k)=\partial g^*(\bar z)$. Since $x_k\in \ri \partial g^*(\bar z)$, we obtain from \eqref{eq:ParC} that 
 \[
 \cone(\partial \hat g^*(-v_k)-x_k)=\cone(\partial  g^*(\bar z)-x_k)={\rm par}\,\partial  g^*(\bar z). 
 \]
 This together with \eqref{eq:Hori} and the fact that $\ker \Bar A=\ker \nabla^2_{xx} f(\bx,\bar p)$ verifies \eqref{con:ND2}.
\end{proof}

When the Fenchel conjugate $g^*$ is $\mathcal{C}^2$-cone reducible, we show next that condition \eqref{con:ND2} is also sufficient for full stability  and tilt stability of problem~\eqref{eq:OP}.

\begin{theorem}[Characterization of full stability and sufficient condition for Lipschitz stability] \label{thm:Lips} Suppose that $\bx$ is an optimal solution of problem \eqref{eq:OP} and that the function $g^*$ is  $\mathcal{C}^2$-cone reducible at $\bar z=-\nabla_x f(\bx,\bar p)\in \partial g(\bx)$. Then the  following are equivalent: 
\begin{itemize}
    \item[{\bf (i)}] $\bx$ is a fully stable optimal solution of problem \eqref{eq:OP}.

    \item[{\bf (ii)}] $\bx$ is a tilt stable optimal solution of problem \eqref{eq:OP}.

    \item[{\bf (iii)}] Condition~\eqref{con:ND2} is satisfied. 
\end{itemize}
Consequently, if additionally $f(\cdot,p)$ is  convex for any $p\in \mathbb{P}$ around $\bar p$, then   the solution mapping 
\begin{equation}\label{eq:Sq}
S(p):={\rm argmin}\left\{f(x,p)+g(x)|\;x\in \bX\right\} 
\end{equation}
is single-valued and Lipschitz continuous around $\bar p$ with $S(\bar p)= \bar x$ provided that condition \eqref{con:ND2} is satisfied. 
\end{theorem}
\begin{proof}
 Suppose that $\bx$ is an optimal solution of problem \eqref{eq:OP}. Note that $\bx$ is also an optimal solution  of problem \eqref{eq:CP2}. The equivalence between {\bf (i)} and {\bf (ii)} follows the arguments discussed  after Theorem~\ref{thm:FS}. Moreover, the implication [{\bf (ii)}$\Rightarrow${\bf (iii)}] is obtained by Theorem~\ref{thm:NeTi}.   It remains to verify [{\bf (iii)}$\Rightarrow${\bf (ii)}]. Indeed, suppose that condition \eqref{con:ND2} is satisfied and consider the following linearly  perturbed version of \eqref{eq:CP2}:
\begin{equation}\label{eq:pTilt}
 \min_{x\in \bX}\quad \frac{1}{2}\|\Bar Ax-\bar b\|^2+\tilde g(x) -\ip{v}{x}
\end{equation}
with tilt parameter $v\in \bX$. As $ \partial\tilde g^*(0)=\partial g^*(\bar z)$,  condition~\eqref{con:ND2} is equivalent to 
\[
\ker \Bar A\cap {\rm par}\,\partial \tilde g^*(0)=\{0\}. 
\]
Since $g^*$ is  $\mathcal{C}^2$-cone reducible at $\bar z$, we infer that  $\tilde g^*=g^*((\cdot)-\bar z)-\ip{\bar z}{\bar x}$ is  $\mathcal{C}^2$-cone reducible at $0$. Applying Theorem~\ref{thm:Sufi} to problem~\eqref{eq:pTilt} tells us that its solution mapping with variable $v$ is single-valued and Lipschitz continuous around $0$, i.e., $\bar x$ is a tilt stable solution of problem \eqref{eq:pTilt}. By Theorem~\ref{thm:FS}, we have 
\[
\|\Bar Aw\|^2+\inf\{\ip{z}{w}|\; z\in D^*\partial \tilde g(\bar x|\, 0)(w)\}>0\quad \mbox{for all}\quad w\in \bX\setminus\{0\}
\]
as in \eqref{eq:Ati}. Note also that $\|\Bar Aw\|^2=\ip{\Bar A^*\Bar A w}{w}=\ip{\nabla_{xx}^2 f(\bx,\bar p)w}{w}$. The above condition is equivalent to \eqref{Tilt}, which is exactly the condition \eqref{Tilt} or \eqref{eq:Tilt}. By Theorem~\ref{thm:FS}, $\bar x$ is a tilt stable optimal solution of problem~\eqref{eq:OP}, which means {\bf (ii)} is satisfied.

Finally, suppose that the function $f(\cdot,p)$ is convex for any $p\in \mathbb{P}$ around $\bar p$ and condition~\eqref{con:ND2} is satisfied. Hence, $\bar x$ is a fully stable optimal solution. It follows that there exists some $\gam>0$ such that the mapping 
\[
M_\gam(v,p):=\argmin\left\{f(x,p)+g(x)-\ip{v}{x}|\; x\in \mathbb{B}_\gam(\bx)\right\}
\]
is single-valued and Lipschitz continuous on some neighborhood of $\bp$ with $M_\gam(0,\bp)=\bx$. This allows us to find $\varepsilon>0$ sufficiently small such that $M_\gam(\{0\}\times \mathbb{B}_\varepsilon(\bp))\subset {\rm int}\, \mathbb{B}_\gam(\bx)$. Hence, $M_\gam(0,p)$ for $p\in \mathbb{B}_\varepsilon(\bp)$ is the unique {\em local} optimal solution of 
\begin{equation*}
\min_{x\in \bX}\quad f(x,p)+g(x).
\end{equation*}
As the function $f(\cdot,p)$ is convex for $p\in \mathbb{B}_\varepsilon(\bp)$, $M_\gam(0,p)$ is also the unique global optimal solution. As $S(p)\in M_\gam(0,p)$ for $p\in \mathbb{B}_\varepsilon(\bp)$, we have $S(p)=M_\gam(0,p)$. Hence, $S(p)$ is single-valued and Lipschitz continuous on $\mathbb{B}_\varepsilon(\bp)$. The proof is completed. 
\end{proof}

Let us conclude this section by providing a simple example showing that  condition \eqref{con:ND2} is  not necessary for  the Lipschitz stability of the solution mapping $S(p)$ in \eqref{eq:Sq}. This essentially highlights the difference between Theorem~\ref{thm:Full} and Theorem~\ref{thm:Lips}, where the function $f$ in Theorem~\ref{thm:Full} has a special structure as a quadratic function. 

\begin{example}{\rm Consider the following simple optimization problem
\begin{equation}\label{eq:Ex1}
\min_{x\in \R}\quad x^4p,
\end{equation}
where $p\in \R$ is a parameter around $\bar p=1$,  $f(x,p)=x^4p$, and $g(x)=0$. The solution set $S(p)=\{0\}$ is single-valued and Lipschitz continuous around $\bar p$. At $\bar p$, the solution of problem \eqref{eq:Ex1} is $\bx=0$.  
Note that   $\bar z=-\nabla_x f(\bar x,\bar p)=0$, $g^*(v)=\delta_{0}(v)$ for $v\in \R$, and $\nabla^2_{xx} f(\bx,\bar p)=0$. It follows that $\partial g^*(\bar z)=\R$ and $\ker \nabla^2_{xx} f(\bx,\bar p)=\R$, which implies
\[
\ker \nabla^2_{xx} f(\bx,\bar p)\cap {\rm par}\, \partial g^*(\bar z)=\R. 
\]
Conditions \eqref{con:ND2}   fails in this case. 
}
\end{example}
\section{Conclusion}
In this paper, we provide necessary and sufficient conditions for the Lipschitz stability of the solution mapping \eqref{eq:Sp} of regularized least-squares optimization problem in the format of \eqref{eq:P}. When the Fenchel conjugate of the regularizer $g$ is $\mathcal{C}^2$-cone reducible, we show that the Lipschitz stability of the solution mapping \eqref{eq:Sp} is equivalent to our condition~\eqref{eq:ND}. This result recovers the corresponding findings in \cite{BBH23,N24} when the regularizer $g$ is the $\ell_1$ norm, the $\ell_1/\ell_2$ norm, or the nuclear norm. 

One of the open questions we plan to investigate in the future is the computation of the Lipschitz modulus of the solution mapping when Lipschitz stability occurs. In theory, this Lipschitz modulus may be obtained by calculating the coderivative of the solution mapping \eqref{eq:Sp}; see, e.g., \cite[Theorem~4.10]{M06} and \cite[Theorem 9.39]{RoW98}. However, this approach requires some involved computations of second-order structures of the regularizer $g$, which can be particularly challenging when  $g$ does not have a polyhedral structure. Finding the Lipschitz modulus without relying on second-order information of $g$, as in our approach, remains an open area of research that we wish to continue exploring. 

Another interesting topic that we aim to study is the sensitivity analysis of the solution mapping \eqref{eq:Sp} under our condition~\eqref{eq:ND}. Some initial results in this direction have been established recently in \cite{BBH23,BBH24,BP21,BLPS21,BPS24,VDFPD17}. However, \cite[Example~3.14]{N24} provides an example where the solution mapping \eqref{eq:Sp} is single-valued and Lipschitz continuous but not differentiable when $g$ is the $\ell_1/\ell_2$ norm. Identifying the conditions under which the solution mapping is differentiable is  one of our ongoing projects.

\end{document}